\documentclass[11pt]{article}

\usepackage{vmargin,graphicx,amsmath,amssymb,color,subfigure}
\usepackage{latexsym}

\usepackage{hyperref,pdfsync}

\newtheorem{theorem}{Theorem}[section]
\newtheorem{lemma}[theorem]{Lemma}
\newtheorem{proposition}[theorem]{Proposition}

\def\blacksquare{
\thinspace\nobreak \vrule width 5pt height 5pt depth 0pt}

\newenvironment{proof}{\begin{trivlist}
                       \item[]\hspace{0cm}{\bf Proof: }
                       \hspace{0cm} }{\hfill $\blacksquare$
                     \end{trivlist}}

\makeatletter

\@addtoreset{equation}{section}
\makeatother
\def\d{\; {\rm d}}
\def\bn{\mathbf n}

\def\O{{\cal O}}
\def\div{{{\rm div}\;}}
\def\curl{{{\rm curl}\;}}
\def\gd{{\nabla^\perp}}
\def\cprime{$'$}
\def\loc{{{\rm loc}}}
\def\supp{{{\rm supp}\;}}

\def\R{\mathbb R}

\def\sL{{\rm L}}
\def\sH{{\rm H}}

\def\C{{\cal C}}

\def\K{{\cal K}}

\def\eps{\varepsilon}
\def\epsa{\varepsilon^\alpha}
\def\zea#1{z^{\eps,\alpha}_{#1}}

\def\Kea#1{{\K^{\eps,\alpha}_{#1}}}
\def\dominf#1{\R^2\setminus#1}
\def\Pieam{\Omega^{\eps,\alpha,\mu}}%{\R^2\backslash\cup_{j=1}^{\ne}\Boul{j}}
\def\Pie{\Omega^\eps}%{\R^2\backslash\cup_{j=1}^{\ne}\Boul{j}}
\def\Nea{N_{\eps,\alpha}}
\def\ve{v^\eps}
\def\we{w^\eps}
\def\wek#1{w^{\eps,#1}}

\def\Nx{n_{1}}
\def\Ny{n_{2}}

\def\fe#1{\varphi^{\eps}_{#1}} % pour les fonctions de troncature
\def\ome{\omega^\eps} % pour les fonctions $\omega$
\def\om0{\omega_{0}} % pour les fonctions $\omega$
\def\re{r^\eps}

\def\Kr#1{K_{\R^2}[#1]}

\def\ueps{u^\eps}

\def\Tc{{\cal T}}
\def\Tca#1{{\Tc^{\eps,\alpha}_{#1}}}

\title{Permeability through a perforated domain for the incompressible 2D Euler equations}
\author{V. Bonnaillie-No\"el, C. Lacave and N. Masmoudi}
\def\adrese{
\begin{description}
\item[V. Bonnaillie-No\"el:] IRMAR - UMR6625, ENS Rennes, Univ. Rennes 1, CNRS, UEB,
av Robert Schuman, 35170 Bruz, France.\\
Email: \texttt{bonnaillie@math.cnrs.fr}\\
Web page: {\footnotesize{\texttt{http://w3.bretagne.ens-cachan.fr/math/people/virginie.bonnaillie} }}
\item[C. Lacave:] Universit\'e Paris-Diderot (Paris 7), Institut de Math\'ematiques de Jussieu - Paris Rive Gauche, UMR 7586 - CNRS, B\^atiment Sophie Germain, Case 7012, 75205 PARIS Cedex 13, France.\\
Email: \texttt{lacave@math.jussieu.fr}\\
Web page: \texttt{http://www.math.jussieu.fr/$\sim$lacave/}
\item[N. Masmoudi:] Courant Institute, 251 Mercer St., New York, NY 10012, U.S.A.\\
Email: \texttt{masmoudi@cims.nyu.edu}\\
Web page: \texttt{http://www.math.nyu.edu/faculty/masmoudi/}
\end{description}
}

\begin{document}

\maketitle

\abstract{
We investigate the influence of a perforated domain on the 2D Euler equations. Small inclusions of size $\varepsilon$ are uniformly distributed on the unit segment or a rectangle, and the fluid fills the exterior. These  inclusions  are at least separated by a distance $\varepsilon^\alpha$ and we prove that for $\alpha$ small enough (namely, less than $2$ in the case of the segment, and less than $1$ in the case of the square), the limit behavior of the  ideal fluid does not feel the effect of the   perforated domain at leading order   when $\varepsilon\to 0$. 
} 

%\tableofcontents

\section{Presentation} 

The homogenization of the Stokes operator and of the incompressible Navier-Stokes equations in a porous medium is by now a very classical problem  \cite{Sanchez80,Tartar80,Allaire90a,Mikelic91}. Recently,  more attention was given to the homogenization of other fluid models such as the compressible Navier-Stokes system \cite{Diaz99,Masmoudi02esaim}, the acoustic system  \cite{DAM12} and the incompressible Euler system \cite{MikelicPaoli,LionsMasmoudi,ILL}.   
 
The goal of this paper is to study the effect of small inclusions of size $\varepsilon$ on the behavior of an ideal  fluid governed by the 2D Euler system. One can expect that for very small holes which are well separated, the effect of the   inclusions  disappears at the limit. This is in the spirit of \cite{CM82,Allaire90b} where  critical sizes of the holes  where studied. 

\subsection{The perforated domain}

Let $\K$ be a smooth simply-connected compact set of $\R^2$, which is the shape of the inclusions. More precisely, we assume that $\partial \K$ is a $\C^{1,\alpha}$ Jordan curve. Without loss of generality, we assume that $0\in \stackrel{\circ}{\K}  \subset(-1,1)^2$.  Let $\alpha> 0$ and $\mu \in [0,1]$ be two parameters which represent how the inclusions fill the square $[0,1]^2$. 
For $i\geq 1$, $j\geq 1$ and $\varepsilon>0$, we define 
\begin{equation}\label{domain1}
\zea{i,j}:=(\eps+2(i-1)(\eps+\epsa),\eps+2(j-1)(\eps+\epsa))=(\eps,\eps)+2(\eps+\epsa)(i-1,j-1),
\end{equation}
 the centers of the inclusions of size $\eps$: 
\begin{equation}\label{domain2}
\K_{i,j}^{\eps,\alpha}:= \zea{i,j} + \eps \K. 
\end{equation}
The geometrical setting is represented in Figure~\ref{fig.config} (in the case where $\K=\overline{B}(0,1)$).
\begin{figure}[h!t]
\begin{center}
\subfigure[Inclusions along the line ${\cal R}_{\eps,\alpha,0}$\label{fig.configa}]{\includegraphics[height=1.5cm]{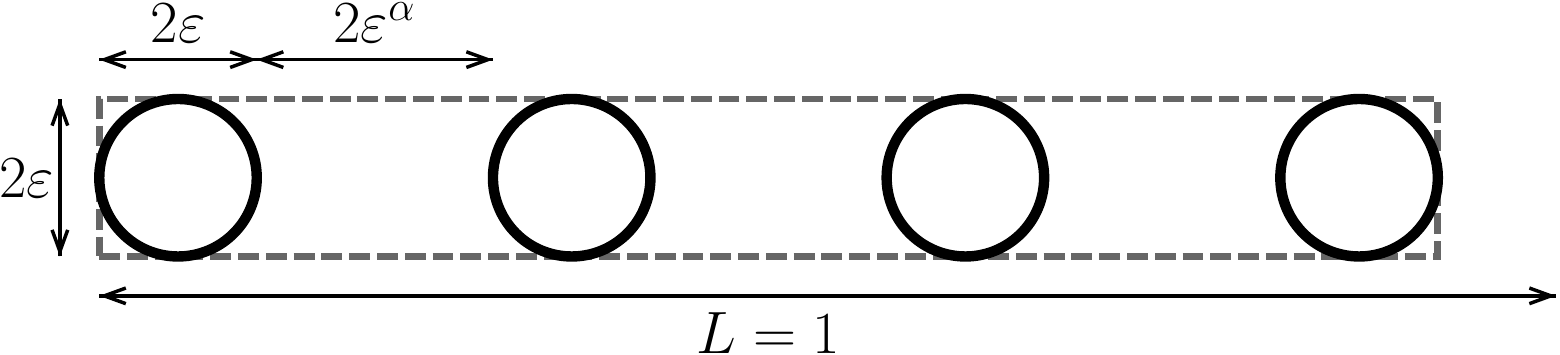}} \hspace{1cm}
\subfigure[Inclusions on the rectangle ${\cal R}_{\eps,\alpha,\mu}$\label{fig.configb}]{\includegraphics[height=4cm]{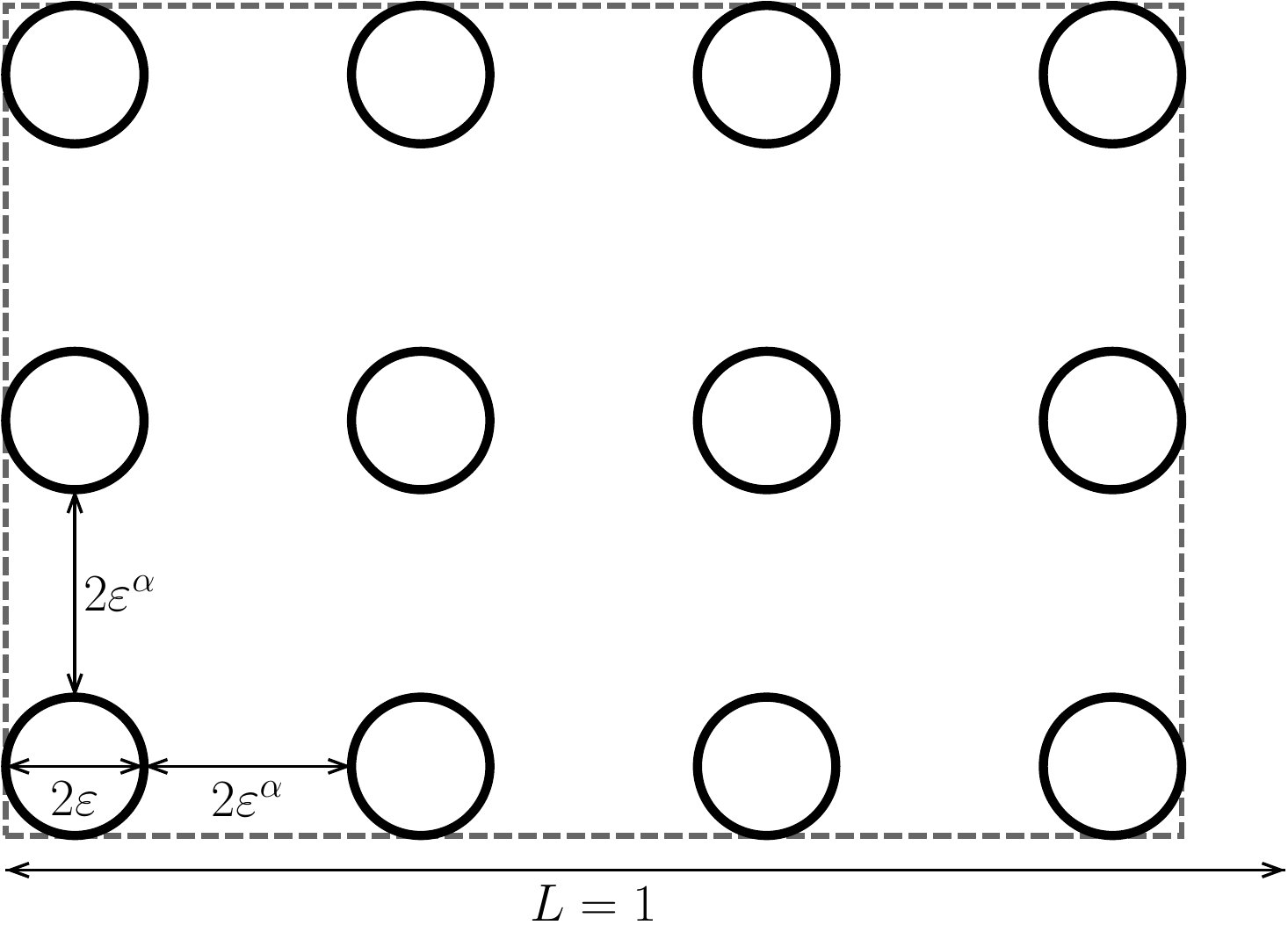}}
\caption{Geometrical settings.\label{fig.config}}%Configuration
\end{center}
\end{figure}

Let $\Nea=\left[\frac{1+2\epsa}{2(\eps+\epsa)}\right]$ (where $[x]$ denotes the integer part of $x$) be the number of inclusions, of size $\eps$ and separated by $2\eps^\alpha$, that we can distribute on the unit segment $[0,1]$ (see Figure \ref{fig.configa}).  
In vertical axis, we assume that there are $\left[(\Nea)^\mu\right]$ inclusions of size $\eps$ at distance $2\epsa$ (with $\mu\in[0,1]$).
For shorter, we denote by $\Nx$ the number of inclusions along the horizontal axis and $\Ny$ those on the vertical axis:
$$\Nx:=\Nea\qquad\mbox{ and }\qquad \Ny:=\left[ (\Nea)^\mu \right].$$
We denote by ${\cal R}_{\eps,\alpha,\mu}$ the rectangle containing all the inclusions: 
\begin{equation}\label{eq.Ream}
{\cal R}_{\eps,\alpha,\mu}=[0,2(\eps+\epsa)\Nx-2\epsa] \times[0,2(\eps+\epsa)\Ny-2\epsa].
\end{equation}
Then the total number of inclusions in ${\cal R}_{\eps,\alpha,\mu}$ equals
\begin{equation}\label{eq.nbincl}
\Nx\Ny \leq (\Nea)^{1+\mu}\leq\left(\frac{1+2\epsa}{2(\eps+\epsa)}\right)^{1+\mu}\leq\frac{1}{(\eps+\epsa)^{1+\mu}},
\end{equation}
as soon as $\eps$ is small enough.

We notice that if $\mu=0$, then $n_2=1$ and there are just inclusions along a line. 
If $\mu=1$, then there are as many inclusions in both directions and in this case the rectangle ${\cal R}_{\eps,\alpha,\mu}$ is almost the square $[0,1]^2$.

We define $\Pieam$ the domain
\begin{equation}\label{domain3}
\Pie=\Pieam:=\dominf{\Bigl(\bigcup_{i=1}^{\Nx}\bigcup_{j=1}^{\Ny}\Kea{i,j}\Bigl)}.
\end{equation}
Since the parameters $\alpha$ and $\mu$ are fixed and we are interested in the limit $\eps\to 0$, the indices $\alpha$, $\mu$ will often be omitted in the notation for shorter.

\subsection{The Euler equations}

Let $\ueps=\ueps(t,x)=(\ueps_{1}(t,x),\ueps_{2}(t,x))$ be the velocity of an incompressible, ideal flow in $\Pie$. 
The evolution is governed by the Euler equations
\begin{equation}\label{eq.Euler}
\left\{\begin{array}{rclcl}
\partial_{t}\ueps+\ueps\cdot\nabla\ueps &=& -\nabla p^\eps & \mbox{ in }& (0,\infty)\times\Pie,\\
\div\ueps &=& 0 & \mbox{ in }& [0,\infty)\times\Pie,\\
\ueps\cdot\bn &=& 0 & \mbox{ in }& [0,\infty)\times\partial\Pie,\\
\lim_{|x|\to \infty}|\ueps(t,x)|&=&0 & \mbox { for }& t\in[0,\infty),\\
\ueps(0,x)&=&\ueps_{0}(x) & \mbox{ in }&\Pie.
\end{array}\right.
\end{equation}
Let $\ome$ be the vorticity defined by
$$\ome:=\curl\ueps = \partial_{1}\ueps_{2}-\partial_{2}\ueps_{1}.$$
The velocity and the vorticity satisfy
\begin{equation}
\label{eq.2}
\left\{\begin{array}{rclcl}
\div\ueps &=& 0 & \mbox{ in }& [0,\infty)\times\Pie,\\
\curl\ueps &=& \ome & \mbox{ in }& [0,\infty)\times\Pie,\\
\ueps\cdot\bn &=& 0 & \mbox{ in }& [0,\infty)\times\partial\Pie,\\
\lim_{|x|\to \infty}|\ueps(t,x)|&=&0 & \mbox { for }& t\in[0,\infty).
\end{array}\right.
\end{equation}

The initial velocity in \eqref{eq.Euler} has to verify:
\begin{equation}\label{initialcond}
\div \ueps_{0} =0 \text{ in } \Pie, \qquad \lim_{|x|\to \infty}|\ueps_{0}(x)|=0, \qquad \ueps_{0}\cdot\bn = 0 \text{ on } \partial \Pie. 
\end{equation}

As our domain depends on $\eps$, it is standard to give the initial data in terms of an initial vorticity independent of $\eps$. Physically, it is relevant to consider the following setting: we assume that the fluid is steady $\ueps_{0}\equiv 0$ at time $t<0$ (then $\ome(0,\cdot)\equiv 0$ and $\ueps_{0}$ has zero circulation around each inclusion) and at time $t=0$ we add, by an exterior force, a vorticity $\omega_{0}$. More precisely, let $\omega_{0}\in \C_{c}^\infty(\R^2)$, then we infer that there exists a unique vector field $\ueps_{0}$ verifying \eqref{initialcond} which has zero circulation around each inclusion and whose  curl is: $\curl \ueps_{0} = \ome_{0} := \omega_{0}\vert_{\Pie}$ (see e.g. \cite{Kikuchi,LLL}).

Then, the vorticity allows us to give an initial condition independent of $\varepsilon$, but the main advantage of the vorticity for the 2D Euler equations comes from the nature of the equations governing the vorticity:
\begin{equation}\label{eq.vort}
\left\{\begin{array}{rclcl}
\partial_{t}\ome+\ueps\cdot\nabla\ome &=& 0 & \mbox{ in }& (0,\infty)\times\Pie,\\
\div\ueps &=& 0 & \mbox{ in }& [0,\infty)\times\Pie,\\
\ueps\cdot\bn &=& 0 & \mbox{ in }& [0,\infty)\times\partial\Pie,\\
\curl\ueps &=& \ome & \mbox{ in }& [0,\infty)\times\Pie,\\
\lim_{|x|\to \infty}|\ueps(t,x)|&=&0 & \mbox { for }& t\in[0,\infty),\\
\oint_{\partial\Kea{i,j}}\ueps(0,s) \cdot \tau\d s&=&0 & \mbox { for  all }&\ i,j     ,\\
\ome(0,\cdot)&=& \omega_{0} &\mbox{ in }&\Pie.
\end{array}\right.
\end{equation}

We can show that the two systems \eqref{eq.Euler} and \eqref{eq.vort} are equivalent, but we obtain more properties from the second system because it is a transport equation. Thanks to this structure, for $\omega_{0}\in \C^\infty_{c}(\R^2)$, Kikuchi establishes in \cite{Kikuchi} that  there exists a unique global strong solution $\ueps$ of \eqref{eq.Euler}, such that $\ome$ belongs to $\sL^\infty(\R^+;\sL^1\cap \sL^\infty(\Pie))$. Actually, for a strong solution $\ueps$, the transport nature of \eqref{eq.vort} implies that:
\begin{itemize}
\item the $\sL^p$ norm of the vorticity is conserved for any $p\in [1,\infty]$:
\begin{equation}\label{est transport}
\|\ome(t,\cdot)\|_{\sL^p(\Pie)}=\|\ome_{0}\|_{\sL^p(\Pie)}\leq \|\omega_{0}\|_{\sL^p(\R^2)},\quad \forall t\geq0,\  \forall p\in[1,+\infty];
\end{equation}
\item the total mass of the vorticity is conserved:
\begin{equation}
\int_{\Pie}\ome(t,x)\d x = \int_{\Pie}\om0(x)\d x;
\end{equation}
\item at any time $t\geq 0$, the vorticity is compactly supported (but the size of the support can grow);
\item the circulation of $\ueps$ around each inclusion is conserved (Kelvin's theorem):
\begin{equation}
\oint_{\partial\Kea{i,j}}\ueps(t,s) \cdot \tau\d s=0 ,\quad \forall t\geq0,\  \forall i,j.
\end{equation}
\end{itemize}

\subsection{Issue and former results}

For $(\alpha,\mu) \in (0,\infty)\times [0,1)$, our domain $\Pie$ converges, in the Hausdorff sense, to the exterior of the unit segment $\R^2 \setminus ([0,1]\times\{0\})$ and for $(\alpha,\mu) \in (0,\infty)\times \{1\}$, it converges to the exterior of the unit square $\R^2 \setminus [0,1]^2$.
Indeed, we check easily that 
\begin{eqnarray*}
d_{H}\left(\bigcup_{j=1}^{\Ny}\bigcup_{i=1}^{\Nx}\Kea{i,j},{\cal R}_{\eps,\alpha,\mu}\right)
&=& \max\left(\sup_{x\in \bigcup_{i,j}\Kea{i,j}}d(x,{\cal R}_{\eps,\alpha,\mu}),\sup_{x\in{\cal R}_{\eps,\alpha,\mu}}d(x, \bigcup_{i,j}\Kea{i,j})\right)\\
&\leq&\max(0,\sqrt 2(\eps+\epsa)),
\end{eqnarray*}
and 
$$\begin{cases}
d_{H}({\cal R}_{\eps,\alpha,1},[0,1]^2) \leq \max(0,2(\eps+\epsa)) & \mbox{ for }\mu=1,\\
d_{H}({\cal R}_{\eps,\alpha,\mu},[0,1]\times\{0\}) \leq \max(2(\eps+\epsa)^{1-\mu},2(\eps+\epsa)) & \mbox{ for }\mu\in[0,1).
\end{cases}$$

{\it The issue of this article is to determine the limit of $(\ueps,\ome)$ when $\varepsilon$ tends to zero, for different values of $\alpha$ and $\mu$, and to compare the limit with the solution in the full plane, or in the exterior of a segment, or in the exterior of a square.}

The well-posedness of the Euler equations in the full plane is well-known since McGrath \cite{McGrath}. In the exterior of a sharp domain, let us mention that the existence of a global weak-solution to the Euler equations in the exterior of the segment, such that $\omega_{0} \in \sL^\infty(\R^+;\sL^1\cap \sL^\infty(\R^2\setminus  ([0,1]\times\{0\})))$, is established  in \cite{Lac-curve}. Such a result is recently extended to the exterior of any connected compact set in \cite{GV-Lac}, for example outside the unit square.

Physically, we can preview that we do not feel the presence of the inclusions for small $\alpha$ (i.e. $\varepsilon^\alpha \gg \varepsilon$) and for small $\mu$, whereas it should appear a wall for $\mu\in [0,1)$ and $\alpha$ large, and the unit square for $\mu=1$ and $\alpha$ large. Moreover, we can think that the critical $\alpha$ should be a decreasing function in terms of $\mu$.

The study of the Euler equations in the exterior of one small obstacle was initiated by Iftimie, Lopes Filho and Nussenzveig Lopes in \cite{ILL}. In that paper, the authors consider only one obstacle which shrinks homotetically to a point, and indeed, if the initial circulation is zero, then their result reads as the solution $(\ueps,\ome)$ converges to the solution in the full plane. Later Lopes Filho has treated in \cite{Lopes} the case of several obstacles in a bounded domain when one of them shrinks to a point. The final result is the same: if initially the circulation is zero, we do not feel the presence of the point at the limit. Finally the last generalization can be found in \cite{LLL} where an infinite number of obstacles is considered. We quote here the theorem in the case where all the initial circulations are equal to zero:
\begin{theorem} \label{theo LLL}
Let $\omega_0\in \C^\infty_{c}(\R^2)$. Let us also fix $R_0>0$ such that ${\rm supp}\  \omega_{0} \subset B(0,R_0)$. For any sequences $\{ z_i^{k}\}_{i=1\dots n_k}\in B(0,R_{0})^{n_k}$, there exists a subsequence, again denoted $k$, and a sequence  $ \eps_{k}\in \R^+_*$ tending to zero such that the solutions $(u^{k} ,\omega^{k})$ of \eqref{eq.vort} in 
\[\Omega^{k}:= \R^2\setminus \Bigl(\bigcup_{i=1}^{n_k} \overline{B}(z_i^k,\eps_k)\Bigl),\]
with initial vorticity $\omega_{0}\vert_{\Omega^k}$ and initial circulations $0$ around the balls, verify
\begin{itemize}
\item[(a)] $ u^{k} \to u$ strongly in $\sL^p_{\loc}(\R_+\times\R^2)$ for any $p\in [1,2)$;
\item[(b)] $ {\omega}^{k} \rightharpoonup {\omega}$ weak $*$ in $\sL^\infty(\R_+; \sL^{q}(\R^2))$ for any $q\in [1,\infty]$;
\item[(c)] the limit pair $(u,{\omega})$ is the unique solution of the Euler equations in the full plane, with initial vorticity $\omega_{0}$.
\end{itemize}
\end{theorem}
In that theorem, we have extended $u^k$ and $\omega^k$ by zero in $(\Omega^{k})^c$. Therefore, we could consider $z_{i,j}^\eps$ as in our configuration (see the first subsection), however there is no control on $\eps_k$ in terms of the distance between the points. The size of the ball can be very small compare to this distance (i.e. $\alpha \ll 1$), and the goal of this article is to get this control. Let us mention that all the works cited before \cite{GV-Lac,ILL,Lac-curve,LLL,Lopes} consider also non-zero initial circulations, and in particular around small obstacles \cite{ILL,LLL,Lopes} the authors find a reminiscent term which appears from the vanishing obstacles. Removing the assumption of zero initial circulations in the present work could be the subject of a future research.

Before stating our result, we also mention a work with the opposite result. The third author and Lions have treated in \cite{LionsMasmoudi} a case which is close to our configuration with $(\alpha,\mu)=(1,1)$. We write ``close'' because that article considers bounded domains $[0,1]^2\setminus \Bigl(\bigcup_{j=1}^{\Ny}\bigcup_{i=1}^{\Nx}\Kea{i,j}\Bigl)$ and the initial condition is not exactly as us. Nevertheless, in the spirit of homogenization and two scale convergence, the authors prove that the limit solution is not the Euler solution in the unit square but rather a  two-scale system that describes the limit behavior. 
In particular the limit solution depends on the shape of the obstacles.

\subsection{Result}

As we can expect, our main result reads as {\it for any $\mu$ there exists a critical $\alpha_c(\mu)$, such that for any $\alpha$ less than $\alpha_c(\mu)$,  the perforated domain is perfectly permeable, i.e. the presence of the inclusions does not perturb the behavior of a perfect fluid.} More precisely:

\begin{theorem} \label{main theo} Let $\Pie$ defined in \eqref{domain1}--\eqref{domain3}, then for all $\mu \in [0,1]$, we define
\[
\alpha_c(\mu)=2-\mu.
\]
Let  $\omega_0$ be a smooth function compactly supported in $\R^2$, $\alpha\in(0, \alpha_{c}(\mu))$ and any sequence  $\varepsilon \to 0$, then the solutions $(\ueps ,\ome)$ of \eqref{eq.vort} in $\Pie$ with initial vorticity $\omega_{0}\vert_{\Pie}$ and initial circulations $0$ around the inclusions, verify:
\begin{itemize}
\item[(a)] $\ueps \to u$ strongly in $\sL^2_{\loc}(\R^+\times\R^2)$;
\item[(b)] $ \ome \rightharpoonup {\omega}$ weak $*$ in $\sL^\infty(\R^+\times\R^2)$;
\item[(c)] the limit pair $(u,{\omega})$ is the unique global solution to the Euler equations in the full plane $\R^2$, with initial vorticity $\omega_{0}$.
\end{itemize}
\end{theorem}
Again, in the previous theorem and in all the sequel, we extend $(\ueps,\ome)$ by zero inside the inclusions.

We note that the function $\mu\mapsto \alpha_c(\mu)$ is continuous, decreasing, positive, such that $\alpha_c(0)=2$ and $\alpha_c(1)=1$. Finally, this result does not depend on the shape of the inclusions.

\medskip

Even if zero circulation are treated in the previous theorem, the goal here is to investigate the effect of the ratio distance/size of the inclusions, an important parameter not controlled in \cite{LLL}. Such a question is investigated by the first author on some elliptic problems such that the Laplace and Navier equations in \cite{BDTV09,BBDHTV11,BD13}, and we emphasize that the Euler equations are linked to such a problem. Indeed we already have a good control of the $\sL^p$ norm of the vorticity, and the velocity can be deduced from the vorticity by a kernel of the 
  type $\nabla^{\perp} \Delta^{-1}$.

We note that a possible extension can be made considering a  less regular $\omega_{0}$, 
  belonging to the space  $\sL^1 \cap \sL^p(\R^2)$ for some $p>2$.

An important future work will be to prove that this $\alpha_{c}(\mu)$ is well critical, in the sense that we note a non negligible effect from the inclusions if $\alpha\geq \alpha_{c}(\mu)$. In fact, the result in \cite{LionsMasmoudi} is already a first hint that 
  it is well the case at least in the case $\mu =1$ and for any type of obstacles. 

Our result should be compared  with critical values obtained  with other equations. The study of the behavior of a flow through a porous medium has a long story in the homogenization framework. The most common setting is to consider a bounded domain $\Omega$ containing many tiny solid obstacles, distributed in each direction. For the Stokes equations with Dirichlet boundary condition, 
 Cioranescu and Murat considered  the case where the ratio $R^\eps := $ (size of the inclusions)/(distance) is $e^{-1/\varepsilon^2}/\varepsilon$, and they obtained  in \cite{CM82} that the limit equation   
  contains an additional term due to the holes. Concerning Stokes and Navier-Stokes, Allaire extensively treated  the previous problem, for e.g. in \cite{Allaire90b} he showed 
  that if  $R^\eps \ll e^{-1/\varepsilon^2}/\varepsilon $   (the rate  of   Cioranescu-Murat),
  the limit is  the Stokes system (hence we do not feel the presence of the inclusions).  If  $R^\eps  \gg  e^{-1/\varepsilon^2}/\varepsilon $, 
 we get the Darcy law (which was well known in the case where the ratio  is 
 $\varepsilon/\varepsilon$, see references in \cite{Allaire90a}). 
And if  $R^\eps  =   e^{-1/\varepsilon^2}/\varepsilon $,  we get 
 the Brinkman type law. Therefore, the above study has treated every case for the viscous problem, and we note that the critical rate $e^{-1/\varepsilon^2}/\varepsilon$ is very small compared  to $\varepsilon/\varepsilon$, which is the rate obtained in the case of the square ($\mu=1$, $\alpha=\alpha_{c}(1)=1$). 

However, an important question is to understand what is the role of the viscosity in the determination of the critical rate (see \cite{MikelicPaoli} for more motivation). For a modified Euler equations, Mikeli\'c-Paoli \cite{MikelicPaoli} and Lions-Masmoudi \cite{LionsMasmoudi} consider a bounded domain perforated in both direction where the rate is $\varepsilon/\varepsilon$, and the limit homogenized system takes into account of the inclusions. Our result is complementary of these articles.

There are also many works concerning inclusions distributed on the unit segment (through grids, sieves or porous walls, we refer e.g. to Conca-Sep{\'u}lveda \cite{conca} and Sanchez-Palenlencia \cite{SP1}). In this setting, the study of the Stokes and Navier-Stokes system is performed by Allaire \cite{Allaire90b}, where he obtained the similar result than before, except that the critical rate is $e^{-1/\varepsilon}/\varepsilon$, which is naturally bigger than $e^{-1/\varepsilon^2}/\varepsilon$, but which stays to be very small compared to our rate: $\varepsilon/\varepsilon^2$ ($\mu = 0$, $\alpha=\alpha_{c}(0)=2$).

\subsection{Plan of the paper}

Thanks to the transport nature of the equation governing the vorticity, we will deduce easily from \eqref{est transport} the point {\it (b)} of Theorem \ref{main theo} from the Banach Alaoglu theorem. In the sequel, we keep the notation $\varepsilon$ even if we extract a subsequence. Indeed, as the limit pair is unique, we will be able to conclude that the limit is the same for any subsequence, so for the full sequence.

The difficulty is to prove {\it (a)}, i.e. that $\ueps=\ueps[\ome]$ converges to $u$ with
\begin{equation}\label{Kr}
u(x):=\Kr{\omega}(x)=\frac1{2\pi} \int_{\R^2}\frac{(x-y)^\perp}{|x-y|^2}\omega(y) \d y,\quad\forall x\in \R^2.
\end{equation}
This formula is the well-known Biot-Savart law in the full plane, i.e. which gives the unique vector field in $\R^2$ which is divergence free, tending to zero at infinity, and whose 
  curl is $\omega$. We need a strong convergence for the velocity in order to pass to 
  the limit in the vorticity equation, and to conclude that the limit pair is well a weak solution of the Euler equations in the full plane. By uniqueness of weak solution (see \cite{yudo}), it will end the proof of {\it (c)} and Theorem \ref{main theo}.

\medskip
 
The main idea is to introduce an explicit  modification of 
  $\Kr{\ome}$, denoted by $\ve[\ome]$,  in order to have a tangent vector field in $\Pie$ whose   curl is $\ome$ plus a small error term. In Section \ref{sect 2}, we recall the explicit formula of the Biot-Savart law in the exterior of one obstacle $\K$, thanks to the Riemann mapping which sends $\K^c$ to $\overline{B}(0,1)^c$, and we present a construction of this 
 modification, based on some cut-off functions around each inclusion.

Then we will write the decomposition:
\begin{equation}\label{decomp}
\begin{array}{rlccccc}
\ueps - u &= &\Bigl(\ueps[\ome]- \ve[\ome]\Bigl)&+&\Bigl(\ve[\ome]-\Kr{\ome} \Bigl) &+& \Kr{\ome-\omega}\ \\
&=:&\re[\ome]&-&\we[\ome]&+& \Kr{\ome-\omega}.
\end{array}
\end{equation}

The central part of this article will be Section \ref{sect fixed time}: at time $t$ fixed, we will look for the critical value of $\alpha$ (in terms of $\mu$), below which the convergence of $\we$ to zero in $\sL^2(\R^2)$ holds. This will follow from a careful study of the explicit formula. Next, we will simply note that $\re$ is the Leray projection 
  of $\we$. As this projector is orthogonal in $\sL^2$, this will give the convergence of $\re$ to zero in $\sL^2(\R^2)$. Thanks to these two convergences, we will prove in Section \ref{sect conclusion} the main theorem.

In the sequel,   $C$ will denote a constant independent of the underlying parameter (which will often be $\varepsilon$), the value of which can possibly change from a line to another.

\section{Explicit formula of the correction}\label{sect 2}

In $\Pie$, we note that $\ueps$ (solving \eqref{eq.2} and having zero circulation around each inclusion) and $\Kr{\ome}$ (see \eqref{Kr}) are divergence free, with the same curl and the same limit at infinity. Moreover, they have the same circulations because we compute by the Stokes formula that 
$$\oint_{\partial\Kea{i,j}}\Kr{\ome}(s) \cdot \tau\d s= \int_{\Kea{i,j}}\ome(x)\d x=0.$$ 
The only differences are that $\Kr{\ome}$ is not tangent to $\partial \Kea{i,j}$, and that we do not have an explicit formula of $\ueps$ in terms of $\ome$. The goal of this section is to correct this lack of tangency.

In this section, we fix the time $t$, i.e. we consider $f$ as a function depending only on $x\in \R^2$, belonging in $\sL^1\cap \sL^\infty(\R^2)$ whose  support is bounded.

\subsection{The Biot-Savart law in an  exterior domain}

In the full plane, we know that there is a unique vector field $u$ satisfying in $\R^2$:
\[
\div u =0, \qquad \curl u= f, \qquad \lim_{|x|\to \infty} |u(x)|=0,
\]
which is given by the standard Biot-Savart formula:
\begin{equation}\label{BS plane}
u(x)=\Kr{f}(x):=\frac1{2\pi} \int_{\R^2}\frac{(x-y)^\perp}{|x-y|^2}f(y) \d y=\frac1{2\pi} \nabla^\perp  \int_{\R^2}\ln|x-y| f(y) \d y,\quad\forall x\in\R^2.
\end{equation}
It is also well known (see e.g. \cite{MajdaBertozzi}) that there is a universal constant $C$ such that
\begin{equation}\label{est biot}
\| \Kr{f} \|_{\sL^\infty(\R^2)} 
\leq \Big\|\frac{1}{2\pi} \int_{\R^2}\frac{|f(y)|}{|x-y|} \d y\Big\|_{\sL^\infty(\R^2)}
\leq C \|f\|_{\sL^1(\R^2)}^{1/2} \|f \|_{\sL^\infty(\R^2)}^{1/2},
\end{equation}
and if $f$ is compactly supported, we have the following behavior at infinity :
\[
 \Kr{f}(x)=\frac{\int_{\R^2} f(y)\d y}{2\pi}\ \frac{x^\perp}{|x|^2}+  \mathcal{O}\left(\frac1{|x|^2}\right).
\]
We note here that considering $u_{0}=\Kr{\omega_{0}}\in \sL^2(\R^2)$ is too restrictive because it would imply that $\int \omega_{0}=0$.

\medskip

In the exterior of a unit disk in dimension 2, we have again an explicit formula for the Biot-Savart law: there exists a unique vector field ${\mathsf u}[f]$ solving in $\R^2\setminus B(0,1)$:
\begin{equation*}
\begin{split}
\div {\mathsf u}[f] =0 , \qquad \curl {\mathsf u}[f] =f , \qquad \lim_{|x|\to \infty}|{\mathsf u}[f](x)|=0, \\
 {\mathsf u}[f]\cdot\bn\vert_{\partial B(0,1)} = 0  , \qquad  \oint_{\partial B(0,1)}{\mathsf u}[f](s) \cdot \tau\d s =0.
\end{split}
\end{equation*}
This vector field ${\mathsf u}[f]$ is given explicitly by:
\begin{equation*}
\begin{split}
 {\mathsf u}[f](x) =&\frac{1}{2\pi} \int_{B(0,1)^c} \Bigg( \frac{x-y}{|x-y|^2}-\frac{x-y^*}{|x-y^*|^2}\Bigg)^{\perp} f(y) \d y + \frac{\int_{B(0,1)^c} f(y) \d y}{2\pi} \frac{x^\perp}{|x|^2}\\
 =& \frac{1}{2\pi} \nabla^\perp \int_{B(0,1)^c} \ln \frac{|x-y| |x|}{|x-y^*|} f(y) \d y,
\end{split}
\end{equation*}
with the notation $z^*=z/|z|^2$ (coming from the image method in order to have a tangent vector field). As we have mentioned in the introduction, solving the elliptic equation \eqref{eq.2} is equivalent to solving  $\Delta \psi = f$, where  $\psi$ is  constant on  the boundary (here, the boundary has  only one connected component) and setting $u:=\nabla^\perp \psi$. Hence, the previous Biot-Savart law comes from the explicit formula of the Green's function in $B(0,1)^c$. Another advantage of the dimension two is that we can extend this formula to the exterior of any simply-connected compact set $\K$: thanks to the complex analysis (identifying $\R^2$ and $\mathbb{C}$) and the fact that holomorphic function is a good change of variable for the Laplace problem. By the Riemann mapping theorem, there exists a unique biholomorphism $\Tc$ mapping $\K^c$ to $\overline{B}(0,1)^c$ and verifying $\Tc(\infty)=\infty$ and $\Tc'(\infty)\in \R^+$. The last condition reads in the Laurent decomposition of $\Tc$ at infinity:
\begin{equation*} 
\Tc(z) = \beta z + \gamma + \O_{z\to\infty}\Big(\frac 1z\Big), \quad\mbox{ with }\quad \beta \in \R^+.
\end{equation*}
Then, we will use several times that
\begin{equation}\label{T expansion}
\Tc(z) = \beta z  +h(z),
\end{equation}
where $h$ is an holomorphic function satisfying at infinity $h(z)=\O(1)$ and $h'(z)=\O(1/|z|^2)$. Of course we have a  similar behavior for $\Tc^{-1}$.\\
In the sequel, we will need  a kind of mean value theorem in a non convex domain given by the following lemma:
\begin{lemma}\label{lem.TT-1}
We assume that $\K$ is a compact set such that $\partial\K$ is a $\C^{1,\alpha}$ Jordan curve. There exists $C$ such that 
\begin{equation*}
\begin{split}
|\Tc(x)-\Tc(y)|\leq C|x-y|,&\qquad\forall (x,y)\in(\K^c)^2,\\
|\Tc^{-1}(x)-\Tc^{-1}(y)|\leq C|x-y|,&\qquad\forall (x,y)\in(\overline B(0,1)^c)^2.\\
\end{split}
\end{equation*}
\end{lemma}
\begin{proof}
As long as the boundary is ${\cal C}^{1,\alpha}$, we can extend the definition of $\Tc$ and $D\Tc$ continuously up the boundary due to Kellogg-Warschawski theorem (see \cite[Theo. 3.6]{Pomm}). Hence, by the behavior at infinity (see \eqref{T expansion}), we infer that $D\Tc$ is uniformly bounded on $\K^c$. The same argument gives also that $\Tc^{-1}$ is bounded on $\overline B(0,1)^c$.

By the connectivity of $\K^c$, we know that for any $x,y\in \K^c$, there exists a smooth path $\gamma$ in $\K^c$ joining $x$ and $y$, and we have
\begin{equation*}
|\Tc(x)-\Tc(y) |
=\Big|\int_{0}^1 D\Tc(\gamma(t))\gamma'(t)\d t \Big|
\leq \|D\Tc\|_{\sL^\infty} \ell(\gamma).
\end{equation*}

Therefore, it is sufficient to prove that there exists $a\geq 1$ such that $\K^c$ is {\it a-quasiconvex}, that is, for all points $x,y$ there exists a rectifiable path $\gamma$ joining $x,y$ and satisfying
\[
\ell(\gamma) \leq a |x-y|.
\]

We note easily that $\overline B(0,1)^c$ is $\frac{\pi}2$-quasiconvex which ends the proof for $\Tc^{-1}$.

Concerning $\Tc$, we remark that $\K^c$ cannot be quasiconvex if $\partial \K$ has a double point or a cusp. Conversely, if $\partial \K$ is a $\C^1$ Jordan curve, it is rather classical to show that there exists $a\geq 1$ such that $\K^c$ is $a$-quasiconvex. We refer to Hakobyan and Herron \cite{Hakobyan} for recent development about quasiconvexity. This kind of problem is although extensively study in complex analysis, and Ahlfors shows in \cite{Ahlfors} the following equivalence in dimension two:
\[
\partial \K \text{ is a quasidisk} \Longleftrightarrow \K^c \text{ is quasiconvex}
\]
where it is known that a Jordan curve, piecewise $\C^1$, is a quasidisk iff $\partial \K$ has no cusp (see e.g. \cite{Gustafsson}).
\end{proof}

Next, with the definitions \eqref{domain1}--\eqref{domain2}, we set $\Tca{i,j}$ as
\begin{equation}\label{Tca}
\Tca{i,j}(z) = \Tc\Bigg( \frac{z-\zea{i,j}}{\eps}\Bigg),
\end{equation}
 the unique biholomorphism which maps $(\Kea{i,j})^c$ to $\overline{B}(0,1)^c$ and  
 satisfies  $\Tca{i,j}(\infty)=\infty$ and $(\Tca{i,j})'(\infty)\in \R^+$.
Let us note that
\begin{equation}\label{Tca-1}
(\Tca{i,j})^{-1}(z) = \eps \Tc^{-1}(z) +\zea{i,j}.
\end{equation}
From these formulas and Lemma \ref{lem.TT-1}, we will often use the following Lipschitz estimates:
\begin{equation}\label{Lip}
\| \Tca{i,j} \|_{\mathrm{Lip}} \leq \frac{C}{\varepsilon} \quad \text{ and } \quad \| (\Tca{i,j})^{-1} \|_{\mathrm{Lip}} \leq C \varepsilon.
\end{equation}

Then we infer that there exists a unique vector field ${\mathsf u}^\eps_{i,j}[f]$ solving in $\R^2\setminus \Kea{i,j}$:
\begin{equation*}
\begin{split}
\div {\mathsf u}^\eps_{i,j}[f] =0 , \qquad \curl {\mathsf u}^\eps_{i,j}[f] =f , \qquad \lim_{|x|\to \infty}|{\mathsf u}^\eps_{i,j}[f](x)|=0, \\
 {\mathsf u}^\eps_{i,j}[f]\cdot\bn\vert_{\partial \Kea{i,j}} = 0  , \qquad  \oint_{\partial \Kea{i,j}}{\mathsf u}^\eps_{i,j}[f](s) \cdot \tau\d s =0,
\end{split}
\end{equation*}
which is given explicitly by:
\begin{equation}\label{BS ueps}
\begin{split}
{\mathsf u}^\eps_{i,j}[f](x) =& \frac1{2\pi}\nabla^\perp \int_{(\Kea{i,j})^c} \ln\frac{|\Tca{i,j}(x)-\Tca{i,j}(y)|\ |\Tca{i,j}(x)|}{|\Tca{i,j}(x)-\Tca{i,j}(y)^*|}f(y)\d y\\
=&\frac1{2\pi}\nabla^\perp \int_{(\Kea{i,j})^c} \ln\frac{\eps|\Tca{i,j}(x)-\Tca{i,j}(y)|\ |\Tca{i,j}(x)|}{\beta|\Tca{i,j}(x)-\Tca{i,j}(y)^*|}f(y)\d y.
\end{split}
\end{equation}
For more details and literature on this problem, we refer to  \cite[Sect. 2]{ILL}.\\
A useful estimate for the next section is:
\begin{lemma}\label{lem anneau}
There exist $C_{1},C_{2},C_{3},C_{4}$ four positive  numbers such that for all $\varepsilon>0$, $\alpha>0$, $\mu\in [0,1]$, $i\in\{1,\ldots,\Nx\}$, $j\in\{1,\ldots,\Ny\}$,  $r>0$, 
we have 
\[  \Tca{i,j}\Bigl(\partial B(\zea{i,j},r)\cap (\Kea{i,j})^c\Bigr) \subset B\Bigl(0,C_{1}\frac r\eps\Bigr)\setminus B\Bigl(0,C_{2}\frac r\eps\Bigr)\]
and
\[ (\Tca{i,j})^{-1}\Bigl(\partial B(0,r+1)\Bigr) \subset B\Bigl(\zea{i,j},\eps C_{3}(r+1)\Bigr)\setminus B\bigl(\zea{i,j},\eps C_{4}(r+1)\bigr).\]
\end{lemma}
\begin{proof} With the definitions of $\Kea{i,j}$ and $\Tca{i,j}$ (see \eqref{domain2} and \eqref{Tca})
we have to prove that there exist $C_{1},C_{2},C_{3},C_{4}$ such that for all $\varepsilon$ and $r$ we have:
\[ \Tc\Bigr(\partial B\Bigl(0,\frac r\eps\Bigr)\cap \K^c\Bigr) \subset B\Bigl(0,C_{1}\frac r\eps\Bigr)\setminus B\Bigl(0,C_{2}\frac r\eps\Bigr)\]
and
\[  \Tc^{-1}(\partial B(0,r+1)) \subset B(0,C_{3}(r+1))\setminus B(0,C_{4}(r+1)).\]
The second point is obvious, because $\Tc^{-1}$ is a bijection from $B(0,1)^c$ to $\overline{\K^c}$ and as ${\Tc^{-1}(z)}/{z}\to 1/\beta$ as $|z|\to \infty$, we can infer that $z\mapsto{|\Tc^{-1}(z)|}/{|z|}$ has an  upper and lower positive bounds. Indeed, we have assumed that there is a  small neighborhood of zero inside $\K$.

Actually, the first point is the same. Indeed, we are looking for $C_{1},C_{2}$ such that for all $s>0$ we have 
\[ \Tc(\partial B(0,s)\cap \K^c) \subset B(0,C_{1}s)\setminus B(0,C_{2}s).\]
So,  the conclusion comes with the same argument applied to $z\mapsto{|\Tc(z)|}/{|z|}$ where $\Tc$ is a bijection from $\overline{\K^c}$ to $B(0,1)^c$.
\end{proof}

\subsection{Definition and properties of  $v^\eps[f]$}
 
A similar modification  was introduced in \cite{LLL} in the case of a finite number of balls, whose  centers are fixed and whose  radii tend  to zero. Our case is more difficult because the centers change, the shape of the inclusion is more general than a ball and the number of inclusions tends to infinity. The idea is to define $\ve$ such that it is equal to \eqref{BS ueps} in a neighborhood of $\Kea{i,j}$ and to \eqref{BS plane} far away.

For this, let us define some cut-off functions $\fe{i,j}$.
Let $\varphi \in\C^\infty(\R)$ be a positive non-increasing function such that 
$$\varphi (s)=\left\{\begin{array}{rl}
1&\mbox{ if }s\leq1/2,\\
0&\mbox{ if }s\geq 1.
\end{array}\right.$$
We define the cut-off function $\fe{i,j}$ on $\Pie$ by
$$\fe{i,j}(x)=\varphi\left(\frac{1}{\epsa}\left(\|x-\zea{i,j}\|_{\infty}-\eps\right)\right).$$
This function is ${\cal C}^\infty$ almost everywhere and satisfies
\begin{equation*}
0\leq \fe{i,j} \leq 1,\quad \fe{i,j}(x)=\left\{\begin{array}{rcl}
1&\mbox{if}& \|x- \zea{i,j} \|_{\infty} \leq \eps + \frac{\epsa}2,\\
0&\mbox{if}& \|x- \zea{i,j} \|_{\infty} \geq \eps + {\epsa},
\end{array}\right.\end{equation*}
and we recall that $\Kea{i,j} \subset  \{ x\in\R^2,\  \|x- \zea{i,j} \|_{\infty} \leq \eps \}$.
From the definition, we note that
\begin{equation}\label{eq.nablafej}
\|\nabla\fe{i,j}\|_{\sL^\infty(\Pie)}\leq\frac{C}{\epsa},
\end{equation}
\begin{equation}\label{support nablafej}
{\rm meas}({\rm supp}\ \nabla \fe{i,j})
=4(\varepsilon+\epsa)^2-4\Bigl(\varepsilon+\tfrac\epsa2\Bigl)^2
=4\eps^{\alpha+1}+3\eps^{2\alpha}
\leq 4\Bigl(\varepsilon^\alpha(\varepsilon+\varepsilon^\alpha)\Bigl).
\end{equation}
Concerning the support of $\fe{i,j}$ we have 
\begin{equation*}
{\rm meas}({\rm supp}\  \fe{i,j})=4(\varepsilon+\epsa)^2-\varepsilon^2 {\rm meas}(\K),
\end{equation*}
so ${\rm meas}({\rm supp}\  \fe{i,j}) = \O\Bigl(\varepsilon^\alpha(\varepsilon+\varepsilon^\alpha)\Bigl)$ if $\K=[-1,1]^2$, and ${\rm meas}({\rm supp}\  \fe{i,j}) = \O\Bigl(\varepsilon^{2\alpha} + \varepsilon^2 \Bigl)$ if not.
In any case, we have
\begin{equation}\label{support fei}
{\rm meas}({\rm supp}\  \fe{i,j})\leq 4(\varepsilon+\epsa)^2 \leq 8(\varepsilon^2+\eps^{2\alpha}).
\end{equation}
Moreover, we note easily that all the supports are disjoints, i.e. for all $\alpha>0$, $\mu\in(0,1]$, $(i,k)\in\{1,\ldots,\Nx\}^2$ and  $(j,l)\in\{1,\ldots,\Ny\}^2$, we have  
\begin{equation}\label{disjoint supp}
\fe{i,j} \fe{k,l}\equiv 0\quad \mbox{ iff }\quad(i,j)\neq(k,l).
\end{equation}

\bigskip

Now, we can simply define our correction as:
\begin{equation}\label{we formula}
\ve[f]:=\nabla^\perp\psi^\varepsilon[f],
\end{equation}
with
\begin{eqnarray*}
\psi^\varepsilon [f](x)&:=& \frac1{2\pi}\left(1-\sum_{j=1}^{\Ny}\sum_{i=1}^{\Nx}\fe{i,j}(x)\right) \int_{\Pie} \ln|x-y|\ f(y)\d y\\
&&+\frac{1}{2\pi}\sum_{j=1}^{\Ny}\sum_{i=1}^{\Nx} \fe{i,j}(x)\int_{\Pie}{\ln}\frac{\eps|\Tca{i,j}(x)-\Tca{i,j}(y)| |\Tca{i,j}(x)|}{\beta|\Tca{i,j}(x)-\Tca{i,j}(y)^*|}\ f(y)\d y\\
&=&\frac1{2\pi} \int_{\Pie} \ln|x-y|\ f(y)\d y\\
&&- \frac1{2\pi}\sum_{j=1}^{\Ny}\sum_{i=1}^{\Nx}\fe{i,j}(x) \int_{\Pie} {\ln}\frac{\beta |x-y|}{\eps|\Tca{i,j}(x)-\Tca{i,j}(y)| }     \ f(y)\d y\\
&&+\frac{1}{2\pi}\sum_{j=1}^{\Ny}\sum_{i=1}^{\Nx} \fe{i,j}(x)\int_{\Pie}{\ln}\frac{ |\Tca{i,j}(x)|}{|\Tca{i,j}(x)-\Tca{i,j}(y)^*|}\ f(y)\d y.
\end{eqnarray*}

From this definition and the previous subsection, it appears obvious that:
\begin{equation}\label{eq.ve}
\begin{split}
\div \ve[f] =0 \text{ in }\Pie,\qquad  \lim_{|x|\to \infty}|\ve[f](x)|=0,\qquad \\
\ve[f] \cdot\bn\vert_{\partial \Pie} = 0, \qquad  \oint_{\partial \Kea{i,j}}\ve[f](s) \cdot \tau\d s =0, \quad \forall i,j.
\end{split}
\end{equation}
We can also note that the curl of $\ve[f]$ is equal to $f$ in $\Pie$ plus some terms localized on the support of $\nabla \fe{i,j}$. In this article, we do not need to estimate precisely this quantity, so we do not write its expression.

\section{Convergence at fixed time} \label{sect fixed time}

As we have said in the introduction, we want to decompose $\ueps-u$ as in \eqref{decomp} and to pass to limit in each terms. In this section, we fixed the time, i.e. we consider $f$ as a function in $\sL^{\infty}_{c}(\R^2)$. Then, we introduce $\ueps[f]$ such that:
\begin{equation}\label{eq.ue}
\begin{split}
\div \ueps[f] =0 \text{ in }\Pie, \qquad \curl \ueps[f] =f \text{ in }\Pie, \qquad \lim_{|x|\to \infty}|\ueps[f](x)|=0, \\
\ueps[f]\cdot\bn\vert_{\partial \Pie} = 0  , \qquad  \oint_{\partial \Kea{i,j}}\ueps[f](s) \cdot \tau\d s =0,\quad\forall i,j,\qquad
\end{split}
\end{equation}
and $\ve[f]$ the correction of $\Kr{f 1_{\Pie}}$, i.e. $\ve[f]$ given by \eqref{we formula}.

Let $M_{0}>0$ be fixed, the goal here is to prove the convergence of 
$$\we[f]:=\Kr{f1_{\Pie}}-\ve[f] \quad\mbox{ and }\quad \re[f]:=\ueps[f]-\ve[f]$$ 
to zero uniformly in $f$ verifying
\[ \| f \|_{\sL^1\cap \sL^\infty(\R^2)} \leq M_{0},\]
where we have extended by zero $\ve[f]$ and $\ueps[f]$ inside the inclusions.

\subsection{Convergence of $\we[f]=\Kr{f1_{\Pie}}-\ve[f]$}

First, in the inclusions, we prove that
\begin{proposition}\label{prop 0}
 For all $p\in [1,\infty)$,
\begin{equation}\label{eq.loin}
\|\we[f]\|_{\sL^p(\dominf{\Pie})}\to 0\quad \mbox{ as }\eps\to 0,\quad 
\forall (\alpha,\mu)\in (0,\infty)\times [0,1)\cup (0,1)\times \{1\},
\end{equation}
uniformly in $f$ verifying
\[ \| f \|_{\sL^1\cap \sL^\infty(\R^2)} \leq M_{0}.\] 
\end{proposition}
\begin{proof}
Indeed, we have  $\Kr{f1_{\Pie}}-\ve[f]=\Kr{f1_{\Pie}}$ on $\dominf{\Pie}$ and by \eqref{est biot} we write that
\begin{eqnarray*}
\|\Kr{f1_{\Pie}}\|_{\sL^p(\dominf{\Pie})} &\leq& \|\Kr{f1_{\Pie}}\|_{\sL^\infty(\R^2)}({\rm meas}(\dominf{\Pie}))^{1/p}\\
&\leq& CM_{0}({\rm meas}(\dominf{\Pie}))^{1/p}.
\end{eqnarray*}
Using~\eqref{eq.nbincl}, we have
\[{\rm meas}(\dominf{\Pie})\leq (\Nea)^{1+\mu}\eps^2{\rm meas} (\K)\leq {\rm meas} (\K) \frac{\eps^2}{(\eps+\epsa)^{1+\mu}}, 
\]
which tends to zero when $\eps\to 0$ for any $\alpha$ if $\mu<1$, and only for $\alpha<1$ if $\mu=1$. Its ends the proof of \eqref{eq.loin}.
\end{proof}

Now, we are working in $\Pie$: using the explicit formula \eqref{we formula}, we decompose as follows
\begin{equation} \label{decompo we}
\we[f] =\sum_{k=1}^4\wek{k}[f],
\end{equation}
with
\begin{equation*}\begin{split}
\wek{1}[f](x)=&\frac1{2\pi} \sum_{j=1}^{\Ny}\sum_{i=1}^{\Nx}\gd \fe{i,j}(x) \int_{\Pie} \ln \frac{\beta|x-y|}{\eps|\Tca{i,j}(x)-\Tca{i,j}(y)|}f(y)\d y, \\
\wek{2}[f](x)=&\frac1{2\pi} \sum_{j=1}^{\Ny}\sum_{i=1}^{\Nx}\gd \fe{i,j}(x) \int_{\Pie} \ln \frac{|\Tca{i,j}(x)-\Tca{i,j}(y)^*|}{|\Tca{i,j}(x)|}f(y)\d y, \\
\wek{3}[f](x)=&\frac1{2\pi} \sum_{j=1}^{\Ny}\sum_{i=1}^{\Nx}\fe{i,j}(x) \int_{\Pie} \Biggl(\frac{(x-y)^\perp}{|x-y|^2}- (D\Tca{i,j})^T(x)\frac{(\Tca{i,j}(x)-\Tca{i,j}(y))^\perp}{|\Tca{i,j}(x)-\Tca{i,j}(y)|^2} \Biggl) f(y)\d y,    \\
\wek{4}[f](x)=&\frac1{2\pi} \sum_{j=1}^{\Ny}\sum_{i=1}^{\Nx}\fe{i,j}(x) (D\Tca{i,j})^T(x) \int_{\Pie} \Biggl(\frac{\Tca{i,j}(x)-\Tca{i,j}(y)^*}{|\Tca{i,j}(x)-\Tca{i,j}(y)^*|^2}- \frac{\Tca{i,j}(x)}{|\Tca{i,j}(x)|^2}\Biggl)^\perp f(y)\d y.
\end{split}\end{equation*}
We prove separately the convergence to $0$  of  each term in $\sL^2$.\\

Let us start by partially dealing with $\wek3$ and $\wek{4}$. Actually, it is very easy if $\mu<1$, without any condition on $\alpha$:
\begin{proposition}\label{prop k}
Let $\mu\in [0,1)$ and $\alpha>0$ be fixed. Then, for $k=3,4$ and any $p\in [1,\infty)$, we have
\begin{equation*}
\|\wek{k}[f]\|_{\sL^p(\Pie)}\to 0\quad \mbox{ as }\eps\to 0,
\end{equation*}
uniformly in $f$ verifying
\[ \| f \|_{\sL^1\cap \sL^\infty(\R^2)} \leq M_{0}.\] 
\end{proposition}

\begin{proof}
Changing variables and using the expression \eqref{Tca} of $\Tca{i,j}$ in terms of $\Tc$, we can get that the quantities
\begin{equation*}\begin{split}
\wek{3}_{i,j}(x):=& \int_{\Pie} \frac{(x-y)^\perp}{|x-y|^2}f(y)\d y - (D\Tca{i,j})^T(x) \int_{\Pie}  \frac{(\Tca{i,j}(x)-\Tca{i,j}(y))^\perp}{|\Tca{i,j}(x)-\Tca{i,j}(y)|^2} f(y)\d y,\\
\wek{4}_{i,j}(x):=& (D\Tca{i,j})^T(x) \int_{\Pie} \Biggl(\frac{\Tca{i,j}(x)-\Tca{i,j}(y)^*}{|\Tca{i,j}(x)-\Tca{i,j}(y)^*|^2}- \frac{\Tca{i,j}(x)}{|\Tca{i,j}(x)|^2}\Biggl)^\perp f(y)\d y,
\end{split}\end{equation*}
are uniformly bounded by $CM_{0}$ where $C$ depends only on $\K$. Indeed, by the Biot-Savart formula, 
$\int_{\Pie} \frac{(x-y)^\perp}{|x-y|^2}f(y)\d y=2\pi K_{\R^2}[f1_{\Pie}]$ and the uniform estimate comes directly from \eqref{est biot}. Concerning the other term of $w_{3}^{i,j}$ and $w_{4}^{i,j}$, all the details are given in \cite[Theorem 4.1]{ILL}.

Hence, as the  $\fe{i,j}$ have  disjoint supports 
  (see \eqref{disjoint supp}), we state that the uniform bound and \eqref{support fei} imply that for any $p\in [1,\infty)$ and $k=3,4$:
\[
\|\wek{k}[f]\|_{\sL^p(\Pie)} 
\leq \frac{3CM_{0}}{2\pi} \Bigl( (\Nea)^{1+\mu} 4(\varepsilon +\epsa)^2 \Bigl)^{1/p} \leq  C M_{0} (\varepsilon +\epsa)^{\frac{1-\mu}{p}},
\]
which tends to zero for $\mu<1$.
\end{proof}

Notice that Proposition \ref{prop k} holds true for any $p\in [1,\infty)$. When $\mu=1$, the proof is more tricky, and we only establish the convergence in $\sL^2$ for $\wek{3}$ and $\wek{4}$ when $\alpha<1$, as we make for $\wek{1},\wek2$.

The terms $\wek1$ and $\wek3$ will be treated in the same spirit. Indeed, 
we note that if $\K=\overline{B}(0,1)$ then $\Tc={\rm Id}$ (so $\beta=1$ and $h=0$ in \eqref{T expansion} in this case) and we would have $\Tca{i,j}(x)=\frac{x-\zea{i,j}}{\eps}$ hence $\wek{1}\equiv 0$ and $\wek{3}\equiv 0$. In the general case, the idea is then to use that $\Tc$ behaves as $\beta{\rm Id}$ at infinity (see \eqref{T expansion}) that justifies the decomposition of the integrals in two parts (close and far away).

\paragraph{Convergence of $\wek{1}$.}
\begin{proposition}\label{prop 1}
We recall  that $\alpha_c(\mu)=2-\mu$. Let $\mu\in [0,1]$ and $\alpha\in (0,\alpha_{c}(\mu))$ be fixed. Then
\begin{equation*}
\|\wek{1}[f]\|_{\sL^2(\Pie)}\to 0\quad \mbox{ as }\eps\to 0,
\end{equation*}
uniformly in $f$ verifying
\[ \| f \|_{\sL^1\cap \sL^\infty(\R^2)} \leq M_{0}.\] 
\end{proposition}

\begin{proof}
We fix $i, j$ and work on the support of $\gd \fe{i,j}$. 
For $x\in\supp \gd \fe{i,j}$ fixed, we decompose the integral in two parts: 
\begin{equation}\label{eq.loinpres}
\begin{split}\Pie_{1}:=\{y\in \Pie,\  |\Tca{i,j}(x)-\Tca{i,j}(y)|\leq \eps^{-1/2}  \},\\
\mbox{ and }\Pie_{2}:=\{y\in \Pie,\  |\Tca{i,j}(x)-\Tca{i,j}(y)|> \eps^{-1/2}  \}.
\end{split}
\end{equation}
In the first subdomain $\Pie_{1}$, we set $z=\eps\Tca{i,j}(x)$ and we change variables $\eta=\eps \Tca{i,j}(y)$ and use \eqref{Tca-1}:
\begin{equation*}\begin{split}
\int_{\Pie_{1}} \Bigl| \ln (\eps|\Tca{i,j}(x)-&\Tca{i,j}(y)|) f(y)\Bigl| \d y\\
&\leq \int_{B(z, \eps^{1/2})} \Bigl| \ln |z-\eta| f((\Tca{i,j})^{-1}(\tfrac\eta\eps))\Bigl| \frac{ \bigl| \det D(\Tca{i,j})^{-1}\bigr|(\tfrac\eta\eps)}{\varepsilon^2}\d \eta\\
&\leq \int_{B(z,\eps^{1/2})} \Bigl| \ln |z-\eta| f((\Tca{i,j})^{-1}(\tfrac\eta\eps))\Bigl|  \bigl| \det D\Tc^{-1}\bigr|(\tfrac\eta\eps) \d \eta.
\end{split}\end{equation*}
Using that $D\Tc^{-1}$ and $f$ are bounded functions, we compute that: 
\begin{equation}\label{eq.w11dom11}
\int_{\Pie_{1}} \Bigl| \ln (\eps|\Tca{i,j}(x)-\Tca{i,j}(y)|) f(y)\Bigl| \d y
\leq M_0 C \int_{B(0, \eps^{1/2} )} \Bigl| \ln |\xi| \Bigl|  \d \xi \leq C M_0 \eps|\ln \eps| .
\end{equation}
To deal with $\ln(\beta|x-y|)$, we remark first that if $y\in \Pie_{1}$, then by \eqref{Lip}, we have
\[
|x-y|=|(\Tca{i,j})^{-1}(\Tca{i,j}(x)) -  (\Tca{i,j})^{-1}(\Tca{i,j}(y))| \leq \eps C |\Tca{i,j}(x)- \Tca{i,j}(y)| \leq C \eps^{1/2}.
\]
Then we compute  that
\begin{eqnarray}
\int_{\Pie_{1}} \Bigl| \ln( {\beta|x-y|})f(y) \Bigl|\d y 
&\leq& \int_{B(x, C\eps^{1/2})} \Bigl| \ln (\beta|x-y|) f(y) \Bigl|\d y \nonumber\\
&\leq& \|f\|_{\sL^\infty}  \int_{B(0, C\eps^{1/2})} \Bigl| \ln |\beta \xi|  \Bigl|\d\xi \nonumber\\
&\leq& CM_{0} \eps |\ln \eps|.\label{eq.w11dom12}
\end{eqnarray}

In the second subdomain $\Pie_{2}$, we have by \eqref{Lip}
\[
\eps^{-1/2} \leq |\Tca{i,j}(x) -  \Tca{i,j}(y)| \leq \eps^{-1} C |x-y|,
\]
hence $|x-y|\geq \frac{ \eps^{1/2}}C$. 
Therefore, with $h$ defined in \eqref{T expansion}, writing
\begin{equation}\label{eq.ln2}
\ln \frac{\eps|\Tca{i,j}(x)-\Tca{i,j}(y)|}{\beta|x-y|}
= \ln \frac{\Bigl|\beta(x-y) + \eps \Bigl(h\big(\frac{x-\zea{i,j}}{\eps}\big)- h\big(\frac{y-\zea{i,j}}{\eps}\big)\Bigr)\Bigr|}{\beta|x-y|},
\end{equation}
we have
\[
\frac{\eps \Bigl|  h\big(\frac{x-\zea{i,j}}{\eps}\big)- h\big(\frac{y-\zea{i,j}}{\eps}\big)\Bigl| }{\beta|x-y|} 
\leq \frac{2C\|h\|_{\sL^\infty}}{\beta }\eps^{1/2},
\]
which is smaller that $1/2$ for $\varepsilon$ small enough.
We note easily that
\begin{equation}\label{est ln}
\Bigl|\ln\tfrac{|b+c|}{|b|} \Bigl| \leq 2 \tfrac{|c|}{|b|},\qquad\mbox{ if }\tfrac{|c|}{|b|}\leq \tfrac12.
\end{equation}
Applying this inequality~\eqref{est ln} with $c=h\big(\frac{x-\zea{i,j}}{\eps}\big)- h\big(\frac{y-\zea{i,j}}{\eps}\big)$ and $b= \tfrac{\beta(x-y)}\eps$, we compute from \eqref{eq.ln2}:
\begin{eqnarray*}
\Biggl|\ln \frac{\eps|\Tca{i,j}(x)-\Tca{i,j}(y)|}{\beta|x-y|}\Biggl|
&\leq& 2\frac{\eps|h(\frac{x-\zea{i,j}}{\eps})-h(\frac{y-\zea{i,j}}{\eps})|}{\beta|x-y|} \leq \frac{4\eps \|h\|_{\sL^\infty}}{\beta|x-y|}.
\end{eqnarray*}
Therefore, using \eqref{est biot}, we obtain
\begin{eqnarray*}
\int_{\Pie_{2}} \Bigl|\ln \frac{\beta|x-y|}{\eps|\Tca{i,j}(x)-\Tca{i,j}(y)|} f(y) \Bigl| \d y
&\leq &\int_{ \Pie}\frac{4\eps \|h\|_{\sL^\infty}}{\beta|x-y|} |f(y)|\d y\\
&\leq &C\eps \| f\|_{\sL^\infty}^{1/2} \| f\|_{\sL^1}^{1/2} \\
& \leq& C M_{0}\eps.
\end{eqnarray*}

Putting together this last estimate with previous ones \eqref{eq.w11dom11}--\eqref{eq.w11dom12}, we get 
\[
\Bigl\| \int_{\Pie} \ln \frac{\beta|x-y|}{\eps|\Tca{i,j}(x)-\Tca{i,j}(y)|}f(y)\d y  \Bigl\|_{\sL^\infty(\Pie)} \leq {C}M_{0}\eps |\ln \eps|,
\]
where $ C$ is a constant independent of $i,j,\alpha,\mu,f,\beta,\eps$. Hence, by \eqref{eq.nablafej}, \eqref{support nablafej}, \eqref{disjoint supp} and since $\eps+\eps^\alpha>\eps$, we finally conclude that
\begin{equation*}\begin{split}
\|\wek{1}[f]\|_{\sL^2(\Pie)} &\leq C M_{0} \frac{\eps |\ln \eps|}{\epsa} \Bigl( N_{\varepsilon,\alpha}^{1+\mu}   \eps^{\alpha}(\eps+\epsa) \Bigl)^{1/2}\leq C M_{0} \frac{\eps |\ln \eps|}{\epsa} \Bigl(   \frac{\eps^{\alpha}}{(\eps+\epsa)^{\mu}} \Bigl)^{1/2}\\
&\leq C M_{0} \frac{\eps |\ln \eps|}{\epsa} \Bigl(   \frac{\eps^{\alpha}}{\eps^{\mu}} \Bigl)^{1/2} = C M_{0}  |\ln \eps| \eps^{\frac{2-\alpha-\mu}{2}},
\end{split}\end{equation*}
which converges to zero if $\alpha<2-\mu$, uniformly in $f$ verifying
\[ \| f \|_{\sL^1\cap \sL^\infty(\R^2)} \leq M_{0}.\] 
\end{proof}

\paragraph{Convergence of $\wek3$.}
\begin{proposition}\label{prop 3}
Let $\mu=1$ and $\alpha\in(0,1)$ be fixed. Then
\begin{equation*}
\|\wek{3}[f]\|_{\sL^2(\Pie)}\to 0\quad \mbox{ as }\eps\to 0,
\end{equation*}
uniformly in $f$ verifying
\[ \| f \|_{\sL^1\cap \sL^\infty(\R^2)} \leq M_{0}.\] 
\end{proposition}

\begin{proof} 
We fix $i, j$ and work in the support of $\fe{i,j}$. 
For $x\in \supp \fe{i,j}$ fixed, we decompose the integral in the two parts defined in \eqref{eq.loinpres}. Using \eqref{Lip}, we have for $y\in\Pie_{1}$
\begin{equation*}
 |x-y|=|(\Tca{i,j})^{-1}(\Tca{i,j}(x)) -  (\Tca{i,j})^{-1}(\Tca{i,j}(y))|
 \leq \eps C |\Tca{i,j}(x)-\Tca{i,j}(y)| \leq C \eps^{1/2},
\end{equation*}
which implies that $\Pie_{1}\subset B(x,C\eps^{1/2})$.
Then we deduce
\begin{equation*}
 \Big| \int_{\Pie_{1}} \frac{(x-y)^\perp}{|x-y|^2}f(y)\d y \Big| 
 \leq \int_{B(x,C\eps^{1/2})} \frac{|f(y)|}{|x-y|}\d y 
\leq C\eps^{1/2} \|f\|_{\sL^\infty} = CM_{0} \eps^{1/2}.
\end{equation*}
In the same way, we deduce from \cite[Theorem 4.1]{ILL} that
\begin{equation*}
\begin{split}
 \Big| (D\Tca{i,j})^T(x) \int_{\Pie_{1}}   \frac{(\Tca{i,j}(x)-\Tca{i,j}(y))^\perp}{|\Tca{i,j}(x)-\Tca{i,j}(y)|^2} f(y)\d y \Big| 
 & \leq C \| f 1_{B(x,C\eps^{1/2})}\|_{\sL^\infty}^{1/2} \| f 1_{B(x,C\eps^{1/2})} \|_{\sL^1}^{1/2}\\
 &\leq C\eps^{1/2} \|f\|_{\sL^\infty} = CM_{0} \eps^{1/2}.
\end{split}
\end{equation*}

For the second subdomain $\Pie_{2}$, we use the expansion \eqref{T expansion} of $\Tc$ to write:
\begin{equation*}
\begin{split}
\frac{(x-y)^\perp}{|x-y|^2}- &(D\Tca{i,j})^T(x)  \frac{(\Tca{i,j}(x)-\Tca{i,j}(y))^\perp}{|\Tca{i,j}(x)-\Tca{i,j}(y)|^2}\\
=&\frac{(x-y)^\perp}{|x-y|^2}- \frac1{\eps}\Big(\beta {\rm Id} + Dh(\tfrac{x-\zea{i,j}}{\eps}) \Big)^T  
  \frac{\Bigl(\beta \frac{x-y}{\eps} + h\big(\frac{x-\zea{i,j}}{\eps}\big)-h\big(\frac{y-\zea{i,j}}{\eps}\big) \Bigr)^\perp}{\Bigl|\beta \frac{x-y}{\eps} + h\big(\frac{x-\zea{i,j}}{\eps}\big)-h\big(\frac{y-\zea{i,j}}{\eps}\big)\Bigr|^2}\\
=& \frac{(x-y)^\perp}{|x-y|^2} 
- \frac{\Big( x-y +\frac{\eps}{\beta} \Big(h\big(\frac{x-\zea{i,j}}{\eps}\big)-h\big(\frac{y-\zea{i,j}}{\eps}\big)\Big) \Big)^\perp}{\Big|x-y + \frac{\eps}{\beta} \Big(h\big(\frac{x-\zea{i,j}}{\eps}\big)-h\big(\frac{y-\zea{i,j}}{\eps}\big)\Big)\Big|^2}\\
&+\frac1{\beta}Dh\big(\tfrac{x-\zea{i,j}}{\eps}\big)^T \frac{\Big( x-y +\frac{\eps}{\beta} (h(\frac{x-\zea{i,j}}{\eps})-h(\frac{y-\zea{i,j}}{\eps})) \Big)^\perp}{|x-y + \frac{\eps}{\beta} (h(\frac{x-\zea{i,j}}{\eps})-h(\frac{y-\zea{i,j}}{\eps}))|^2}\\
=:& J_{1}(x,y)+J_{2}(x,y).
\end{split}
\end{equation*}
Due to \eqref{Lip}, we have for $y\in\Pie_{2}$
\[
\eps^{-1/2} \leq |\Tca{i,j}(x)-\Tca{i,j}(y)| \leq \frac{C}{\eps} |x-y|
\]
and we can deduce  that $\Pie_{2} \subset B(x, \frac{\eps^{1/2}}C)^c$. Furthermore $\frac{\eps}{\beta} |h(\frac{x-\zea{i,j}}{\eps})-h(\frac{y-\zea{i,j}}{\eps})| \leq C \eps$, then, for $\eps$ small enough, we have
\begin{equation*}
|J_1(x,y)|=  \frac{\Big| \frac{\eps}{\beta} (h(\frac{x-\zea{i,j}}{\eps})-h(\frac{y-\zea{i,j}}{\eps})) \Big|}{|x-y| \Big|x-y + \frac{\eps}{\beta} (h(\frac{x-\zea{i,j}}{\eps})-h(\frac{y-\zea{i,j}}{\eps}))\Big|}
 \leq \frac{C\eps }{|x-y| ( \frac{ \eps^{1/2}}C-C\eps)}\leq\frac{C\eps^{1/2}}{|x-y|},
\end{equation*}
where we have used the relation 
\begin{equation}\label{eq.ab}
\left|\frac{a}{|a|^2}-\frac{b}{|b|^2}\right|=\frac{|a-b|}{|a|\ |b|}.
\end{equation} 
Hence we get by \eqref{est biot}
\[
 \Big| \int_{\Pie_{2}} J_{1}(x,y)f(y)\d y \Big| \leq C\eps^{1/2} \int_{\R^2} \frac{|f(y)|}{|x-y|}\d y \leq C\eps^{1/2} M_{0}.
\]
For $J_{2}$, we know that there exists $C$ such that $|z h'(z)| \leq C$ for any $z$ (see \eqref{T expansion}), so
\[
|J_{2}(x,y)| \leq \frac{1}{\beta}\frac{C}{\frac{|x-\zea{i,j}|}\eps} \frac{1}{|x-y|-C\eps}
\leq \frac{C\eps}{|x-\zea{i,j}|}\frac{1}{|x-y|},
\]
hence
\[
 \Big| \int_{\Pie_{2}}J_{2}(x,y)f(y)\d y \Big| 
 \leq \frac{C\eps}{|x-\zea{i,j}|} \int_{\R^2} \frac{|f(y)|}{|x-y|}\d y \leq \frac{C\eps M_{0}}{|x-\zea{i,j}|}.
\]

Putting together the previous estimates, we finally obtain that
\[
  |\wek3_{i,j}(x) |  \leq 3CM_{0} \eps^{1/2}+\frac{C\eps M_{0}}{|x-\zea{i,j}|}.
\]
The $\sL^2$ norm is easy to estimate for the first right hand side term:
\[
\Big\| \sum_{i,j} \fe{i,j} 3CM_{0} \eps^{1/2} \Big\|_{\sL^2(\Pie)} \leq  3CM_{0} \eps^{1/2} \Bigl( (\Nea)^{2} 4(\varepsilon +\epsa)^2 \Bigl)^{1/2} =  C M_{0} \eps^{1/2},
\]
which tends to zero as $\eps\to 0$. Concerning the second right hand side term, we use that $x$ belongs to the support of $\fe{i,j}$ and that there exists $\delta$ such that $B(0,\delta)\subset \K$, hence $x\in B(\zea{i,j}, \sqrt{2}(\eps+\epsa))\setminus B(\zea{i,j},\delta\eps)$. So we compute
\begin{equation*}
\begin{split}
\Big\| \sum_{i,j} \fe{i,j}(x) \frac{C\eps M_{0}}{|x-\zea{i,j}|} \Big\|_{\sL^2(\Pie)} 
&\leq C\eps M_{0}  \Bigl( \sum_{i,j} \int_{B(\zea{i,j}, \sqrt{2}(\eps+\epsa))\setminus B(\zea{i,j},\delta\eps)}  \frac 1{|x-\zea{i,j}|^{2}} \d x \Bigl)^{1/2}\\
&\leq  C\eps M_{0}\Bigl(  (\Nea)^{2} \ln\frac{\sqrt{2}(\eps+\epsa)}{\delta\eps} \Bigl)^{1/2} \leq C M_{0} |\ln \eps|^{1/2} \eps^{1-\alpha},
\end{split}
\end{equation*}
recalling that $\eps<\epsa$ because we have assumed that $\alpha<1$.
This  ends the estimate of  $\wek{3}_{i,j}$: 
\[
\|\wek{3}[f]\|_{\sL^2(\Pie)}\leq CM_{0}\left(\eps^{1/2}+|\ln\eps|^{1/2}\eps^{1-\alpha}\right).
\]
\end{proof}

\bigskip

The general idea to treat $\wek{2}$ and $\wek{4}$ is the following: if $\K = \overline{B}(0,1)$, then $\Tc={\rm Id}$ and $\Tca{i,j}(x)=\frac{x-\zea{i,j}}{\varepsilon}$, so we note that 
\[
\eps | \Tca{i,j}(x)-   \Tca{i,j}(y)^*| = \Big|x-\zea{i,j} + \eps^2 \tfrac{ y-\zea{i,j}}{ |y-\zea{i,j}|^2}\Big| \sim |x-\zea{i,j}| = \eps | \Tca{i,j}(x)|
\]
at least when $|y-\zea{i,j}|>2\eps$. Hence we will also decompose the domains in two subdomains in order to use this hint.

\paragraph{Convergence of $\wek{2}$.}
\begin{proposition}\label{prop 2}
We recall that  $\alpha_c(\mu)=2-\mu$.
Let $\mu\in [0,1]$ and $\alpha\in (0,\alpha_{c}(\mu))$ be fixed. Then
\begin{equation*}
\|\wek{2}[f]\|_{\sL^2(\Pie)}\to 0\quad \mbox{ as }\eps\to 0,
\end{equation*}
uniformly in $f$ verifying
\[ \| f \|_{\sL^1\cap \sL^\infty(\R^2)} \leq M_{0}.\] 
\end{proposition}
\begin{proof}
For $x\in \supp \gd \fe{i,j}$, we set $z=\eps \Tca{i,j}(x)$, and changing variables $\eta = \eps\Tca{i,j}(y)$, we deduce from \eqref{Tca-1} that we need to estimate the following quantity:
\begin{equation}\label{w12}
\wek2_{i,j}(z):=\frac1{2\pi}  \int_{B(0,\eps)^c} \ln \frac{|z- \eps^2 \eta^*|}{|z|}f(\eps \Tc^{-1}(\tfrac\eta\eps)+\zea{i,j}) |\det D\Tc^{-1}|(\tfrac\eta\eps) \d \eta.
\end{equation}

From the definition of the cut-off function, we know that $x\in  \supp \gd \fe{i,j}$ implies that
\[
\eps+\tfrac\epsa2 \leq |x-\zea{i,j}| \leq \sqrt{2}(\eps+\epsa),
\]
then by Lemma \ref{lem anneau},  we deduce  that
\[
C_{2} (\eps+\tfrac\epsa2) \leq |z| \leq C_{1}\sqrt{2}(\eps+\epsa).
\]
Therefore, for any $\eta\in B(0,\eps)^c$ we have
\begin{equation*}
\frac{|\eps^2 \eta^*|}{|z|} \leq \frac{\eps^2}{C_{2}(\eps+\frac{\epsa}2)\ |\eta|}.
\end{equation*}
Hence, using \eqref{est ln} with $b=z$ and $c=-\eps^2 \eta^*$, we infer that we have
\begin{equation}\label{eq.majoln}
\left| \ln \frac{|z- \eps^2 \eta^*|}{|z|} \right|
\leq 2 \frac{\eps^2 |\eta^*|}{|z|}
\leq  \frac{2\eps^2}{C_{2}(\eps+\frac{\epsa}2)\ |\eta|}\qquad \mbox{ if }\quad \frac{\eps^2}{C_{2}(\eps+\frac{\epsa}2) |\eta|}\leq \frac12.
\end{equation}

Keeping in mind this inequality, we define $R=2/C_{2}$ and we split the integral \eqref{w12} in two parts: $ B(0,R\eps)^c$ and $B(0,R\eps)\setminus B(0,\eps)$. 

In the first subdomain $B(0,R\eps)^c$, we use that $\eps+\eps^\alpha>\eps$ and 
$$\frac{\eps^2}{C_{2}(\eps+\frac{\epsa}2) |\eta|}\leq \frac{\eps^2}{C_{2}\eps R\eps}= \frac 12,$$ 
hence by \eqref{eq.majoln}, 
we compute
\begin{equation*}\begin{split}
\Bigl|\int_{B(0,R\eps)^c}& \ln \frac{|z- \eps^2 \eta^*|}{|z|}f(\eps \Tc^{-1}(\tfrac\eta\eps)+\zea{i,j}) |\det D\Tc^{-1}|(\tfrac\eta\eps) \d \eta\Bigl|\\
&\leq  \int_{B(0,R\eps)^c}   \frac{2\eps^2}{C_{2}(\eps+\frac{\epsa}2)\ |\eta|} |f(\eps \Tc^{-1}(\tfrac\eta\eps)+\zea{i,j})|\ |\det D\Tc^{-1}|(\tfrac\eta\eps) \d \eta\\
&\leq \frac{2\eps^2}{C_{2}\eps}  \int_{\R^2} \frac{|f(\eps \Tc^{-1}(\tfrac\eta\eps)+\zea{i,j})|\ |\det D\Tc^{-1}|(\tfrac\eta\eps)}{|\eta|} \d \eta\\
&\leq C\eps  \Big\| f(\eps \Tc^{-1}(\tfrac\eta\eps)+\zea{i,j}) \det D\Tc^{-1}(\tfrac\eta\eps) \Big\|_{\sL^\infty}^{1/2}
\Big\| f(\eps \Tc^{-1}(\tfrac\eta\eps)+\zea{i,j}) \det D\Tc^{-1}(\tfrac\eta\eps) \Big\|_{\sL^1}^{1/2}\\
&\leq C\eps \|f\|_{\sL^\infty}^{1/2}\|f\|_{\sL^1}^{1/2}\leq CM_{0}\eps,
\end{split}\end{equation*}
where we have applied \eqref{est biot} for the function $\eta \mapsto  |f(\eps \Tc^{-1}(\tfrac\eta\eps)+\zea{i,j})|\ |\det D\Tc^{-1}|(\tfrac\eta\eps)$ at $x=0$, used that $D \Tc^{-1}$ is bounded and that $\| f(\eps \Tc^{-1}(\tfrac\eta\eps)+\zea{i,j}) \det D\Tc^{-1}(\tfrac\eta\eps) \|_{\sL^1}=\|f\|_{\sL^1}$ by changing variables back.

In the second subdomain $B(0,R\eps)\setminus B(0,\eps)$, we come back to the original variables: by Lemma \ref{lem anneau},  we compute
\begin{equation}\label{eq.w12extRe}\begin{split}
\Bigl|\int_{B(0,R\eps)\setminus B(0,\eps)} \ln &\frac{|z- \eps^2 \eta^*|}{|z|}f(\eps \Tc^{-1}(\tfrac\eta\eps)+\zea{i,j}) |\det D\Tc^{-1}|(\tfrac\eta\eps) \d \eta\Bigl|\\
&\leq \int_{B(0,C_3R\eps)\setminus \Kea{i,j}} \Bigl| \ln \frac{| \Tca{i,j}(x)-   \Tca{i,j}(y)^*|}{| \Tca{i,j}(x)|} \Bigl| |f(y)| \d y.
\end{split}\end{equation}
Now we note that $\Tca{i,j}(y)^*$ belongs to the unit disk whereas $\Tca{i,j}(x)$ is outside, hence
\[ | \Tca{i,j}(x)| -1  \leq | \Tca{i,j}(x)-   \Tca{i,j}(y)^*| \leq | \Tca{i,j}(x)| +1.\]
Let $X$  be the point of $\partial\Kea{i,j}$ such that $| \Tca{i,j}(x)| -1= | \Tca{i,j}(x) -\Tca{i,j}(X)|$. Then, 
since $x\in \supp \gd \fe{i,j}$, we have $|x-X|\geq |x-\zea{i,j}|-|\zea{i,j}-X|\geq \epsa/2$ and then, with \eqref{Lip}
\[
\frac{\epsa}2 \leq |x-X| \leq \eps C   | \Tca{i,j}(x) -\Tca{i,j}(X)|,
\]
hence,
\[\frac{\eps^{\alpha-1}}{2C | \Tca{i,j}(x)|}   \leq \frac{| \Tca{i,j}(x)-   \Tca{i,j}(y)^*|}{ | \Tca{i,j}(x)|} \leq 1+ \frac{1}{ | \Tca{i,j}(x)|}.\]
Moreover, Lemma \ref{lem anneau} yields  that $ \frac{C_{2}(\eps+\epsa)}{\eps} \leq | \Tca{i,j}(x)| \leq\frac{C_{1}\sqrt{2}(\eps+\epsa)}{\eps}$ so
\[ \frac{\eps^{\alpha}}{2C C_{1}\sqrt{2}(\eps+\epsa)}   \leq \frac{| \Tca{i,j}(x)-   \Tca{i,j}(y)^*|}{ | \Tca{i,j}(x)|} \leq 1+ \frac{\eps}{C_{2}(\eps+\epsa)}\leq 1+\frac 1{C_{2}},\]
which implies that
\[
\Bigg| \ln \frac{| \Tca{i,j}(x)-   \Tca{i,j}(y)^*|}{ | \Tca{i,j}(x)|} \Bigg | \leq C \Bigl(1+ \Bigl| \ln\frac{\eps^{\alpha}}{\eps+\epsa}\Bigr|\Bigr)\leq  C (1+|\ln \eps|).
\]
Therefore, using \eqref{eq.w12extRe},
\begin{equation*}\begin{split}
\Bigl|\int_{B(0,R\eps)\setminus B(0,\eps)} \ln \frac{|z- \eps^2 \eta^*|}{|z|}&f(\eps \Tc^{-1}(\tfrac\eta\eps)) |\det D\Tc^{-1}|(\tfrac\eta\eps) \d \eta\Bigl|\\
&\leq  C (1+|\ln \eps|) \| f \|_{\sL^\infty} \pi (C_{3}R\eps)^2.
\end{split}\end{equation*}

Putting  together the estimates in the  two subdomains we get that $\wek{2}_{i,j}(z)$ is bounded by $C\eps M_0$ uniformly for $x\in \gd \fe{i,j}$. Then we conclude as for $\wek{1}$:
\begin{equation*}\begin{split}
\|\wek{2}[f]\|_{\sL^2(\Pie)} &\leq C M_{0} \frac{\eps }{\epsa} \Bigl( N_{\varepsilon,\alpha}^{1+\mu}   \eps^{\alpha}(\eps+\epsa) \Bigl)^{1/2}\leq C M_{0}  \eps^{\frac{2-\alpha-\mu}{2}},
\end{split}\end{equation*}
which converges to zero if $\alpha<2-\mu$, uniformly in $f$ verifying
\[ \| f \|_{\sL^1\cap \sL^\infty(\R^2)} \leq M_{0}.\] 
\end{proof}

\paragraph{Convergence of $\wek{4}$.}

\begin{proposition}\label{prop 4}
Let $\mu=1$ and $\alpha\in(0,1)$ be fixed. Then
\begin{equation*}
\|\wek{4}[f]\|_{\sL^2(\Pie)}\to 0\quad \mbox{ as }\eps\to 0,
\end{equation*}
uniformly in $f$ verifying
\[ \| f \|_{\sL^1\cap \sL^\infty(\R^2)} \leq M_{0}.\] 
\end{proposition}
\begin{proof}
The idea is the same as  for $\wek{3}$: we compare $\Tca{i,j}(x)-\Tca{i,j}(y)^*$ with $\Tca{i,j}(x)$.
Let us fix $i,j$ and we work on the support of $ \fe{i,j}$. We decompose the integral in two parts $\{y\in \Pie,\  |\Tca{i,j}(y)|\leq 2  \}$ and $\{y\in \Pie,\  |\Tca{i,j}(y)|> 2  \}$. If $y$ verifies $|\Tca{i,j}(y)|\leq 2$, it implies that there exists $\bar y\in \partial \Kea{i,j}$ such that $|\Tca{i,j}(y)-\Tca{i,j}(\bar y)|\leq 1$ (we recall that $\Tca{i,j}$ maps $(\Kea{i,j})^c$ to $\bar{B}(0,1)^c$). Hence, by \eqref{Lip}
\[
|y-\zea{i,j}|\leq |y-\bar y| + |\bar y -\zea{i,j}|\leq C\eps |\Tca{i,j}(y)-\Tca{i,j}(\bar y)| +\sqrt{2}\eps \leq C \eps,
\]
which allows us to estimate in the first subdomain, using \cite[Theorem 4.1]{ILL}:
\begin{equation*}
\begin{split}
\Bigg| (D\Tca{i,j})^T(x) \int_{\{y\in \Pie,\  |\Tca{i,j}(y)|\leq 2  \}}&   \Biggl(\frac{\Tca{i,j}(x)-\Tca{i,j}(y)^*}{|\Tca{i,j}(x)-\Tca{i,j}(y)^*|^2}- \frac{\Tca{i,j}(x)}{|\Tca{i,j}(x)|^2}\Biggl)^\perp f(y)\d y \Bigg| \\
& \leq C \| f 1_{B(\zea{i,j},C\eps)}\|_{\sL^\infty}^{1/2} \| f 1_{B(\zea{i,j},C\eps)} \|_{\sL^1}^{1/2}\\
 &\leq C\eps \|f\|_{\sL^\infty} = CM_{0} \eps .
\end{split}
\end{equation*}

In the second subdomain, we note that $|\Tca{i,j}(y)|> 2$ implies that
\[
|\Tca{i,j}(x) -\Tca{i,j}(y)^*| \geq |\Tca{i,j}(x)| - \frac{1}{|\Tca{i,j}(y)|}\geq \frac12.
\]
As for $\wek{2}$, we set $z=\eps \Tca{i,j}(x)$, and change variables $\eta = \eps\Tca{i,j}(y)$ to obtain with \eqref{eq.ab}:
\begin{equation*}\begin{split}
\Big| (D&\Tca{i,j})^T(x) \int_{\{y\in \Pie,\  |\Tca{i,j}(y)|> 2  \}}   \Biggl(\frac{\Tca{i,j}(x)-\Tca{i,j}(y)^*}{|\Tca{i,j}(x)-\Tca{i,j}(y)^*|^2}- \frac{\Tca{i,j}(x)}{|\Tca{i,j}(x)|^2}\Biggl)^\perp f(y)\d y \Big| \\
=& \Big| (D\Tc)^T(\tfrac{x-\zea{i,j}}{\eps}) \int_{B(0,2\eps)^c}   \Biggl(\frac{z-\eps^2\eta^*}{|z-\eps^2 \eta^*|^2}- \frac{z}{|z|^2}\Biggl)^\perp f(\eps \Tc^{-1}(\tfrac\eta\eps)+\zea{i,j}) |\det D\Tc^{-1}|(\tfrac\eta\eps) \d \eta \Big| \\
\leq & C  \int_{B(0,2\eps)^c}   \frac{\eps^2|\eta^* |}{|z-\eps^2\eta^*|\ |z|}|f(\eps \Tc^{-1}(\tfrac\eta\eps)+\zea{i,j})|\ |\det D\Tc^{-1}|(\tfrac\eta\eps) \d \eta  \\
\leq & \frac{2C\eps}{|z|}  \int_{B(0,2\eps)^c}   \frac{|f(\eps \Tc^{-1}(\tfrac\eta\eps)+\zea{i,j})|\ |\det D\Tc^{-1}|(\tfrac\eta\eps)}{|\eta|} \d \eta,
\end{split}\end{equation*}
so by \eqref{est biot} 
\begin{equation*}\begin{split}
\Big| (D&\Tca{i,j})^T(x) \int_{\{y\in \Pie,\  |\Tca{i,j}(y)|> 2  \}}   \Biggl(\frac{\Tca{i,j}(x)-\Tca{i,j}(y)^*}{|\Tca{i,j}(x)-\Tca{i,j}(y)^*|^2}- \frac{\Tca{i,j}(x)}{|\Tca{i,j}(x)|^2}\Biggl)^\perp f(y)\d y \Big| \\
\leq& \frac{C\eps}{|z|} \| f(\eps \Tc^{-1}(\tfrac\eta\eps)+\zea{i,j}) \det D\Tc^{-1}(\tfrac\eta\eps) \|_{\sL^\infty}^{1/2} \| f(\eps \Tc^{-1}(\tfrac\eta\eps)+\zea{i,j}) \det D\Tc^{-1}(\tfrac\eta\eps) \|_{\sL^1}^{1/2}\\
\leq& \frac{C\eps}{|z|} \| f \|_{\sL^\infty}^{1/2} \| f \|_{\sL^1}^{1/2} \leq \frac{C\eps M_{0}}{|z|},
\end{split}\end{equation*}
where we have changed  variables back. Bringing together the estimates in the two subdomains, we conclude that
\[
|\wek{4}_{i,j}(x)|\leq  CM_{0} \eps + \frac{C\eps M_{0}}{|\eps \Tca{i,j}(x)|} .
\]
As for $|\wek3_{i,j}(x)|$, the first part is easy to estimate in $\sL^2$:
\[
\Big\| \sum_{i,j} \fe{i,j} CM_{0} \eps \Big\|_{\sL^2(\Pie)} \leq  CM_{0} \eps \Bigl( (\Nea)^{2} 4(\varepsilon +\epsa)^2 \Bigl)^{1/2} =  C M_{0} \eps.
\]
Concerning the last part, as there exists $\delta$ such that $\supp \fe{i,j} \subset B(\zea{i,j},\sqrt{2}(\eps+\epsa)) \setminus B( \zea{i,j}, \delta \eps)$, by Lemma \ref{lem anneau} we know that $\eps\Tca{i,j}(x)$ belongs to $B(0,C_{1}\sqrt{2}(\eps+\epsa)) \setminus B( 0, C_{2}\delta \eps)$. Hence we use that $\fe{i,j}$ have disjoint supports and we change variable $z=\eps \Tca{i,j}(x)$:
\begin{equation*}
\begin{split}
\Big\| \sum_{i,j} \fe{i,j}(x) \frac{C\eps M_{0}}{|\eps \Tca{i,j}(x)|}  \Big\|_{\sL^2(\Pie)} 
&\leq C\eps M_{0}  \Bigl( \sum_{i,j} \int_{\supp \fe{i,j}} \frac{1}{|\eps \Tca{i,j}(x)|^{2}} \d x \Bigl)^{1/2}\\
&\leq C\eps M_{0}  \Bigl( \sum_{i,j} \int_{B(0,C_{1}\sqrt{2}(\eps+\epsa)) \setminus B(0, C_{2} \delta \eps)} \frac{1}{ |z|^{2}} \d z \Bigl)^{1/2}\\
&\leq C\eps M_{0}\Bigl(  (\Nea)^{2} \ln\frac{C(\eps+\epsa)}{\eps} \Bigl)^{1/2} \leq C M_{0}|\ln \eps|^{1/2} \eps^{1-\alpha}.
\end{split}
\end{equation*}

Therefore, we have established that 
\[\|\wek{4}[f]\|_{\sL^2(\Pie)} \leq C M_{0} \Big(\eps + |\ln \eps|^{1/2} \eps^{1-\alpha} \Big),\]
which tends to zero as $\eps\to 0$, because we are considering the case $\alpha<1$. Its ends this proof.
\end{proof}

Bringing together all the propositions of this subsection, we have proved the following theorem:

\begin{theorem}\label{theo we}
We recall that  $\alpha_c(\mu)=2-\mu$.
Let $\mu\in [0,1]$ and $\alpha\in (0,\alpha_{c}(\mu))$ be fixed. Then
\begin{equation*}
\|\we[f]\|_{\sL^2(\R^2)}\to 0\quad \mbox{ as }\eps\to 0,
\end{equation*}
uniformly in $f$ verifying
\[ \| f \|_{\sL^1\cap \sL^\infty(\R^2)} \leq M_{0}.\] 
\end{theorem}

\subsection{Convergence of $\re$}\label{sect Leray}

In the decomposition 
\[ \ueps[f] - \Kr{f1_{\Pie}} = \re[f] - \we[f], \]
with
\[
\we[f]=\Kr{f1_{\Pie}}-\ve[f] \qquad \text{and} \qquad \re[f]:=\ueps[f]-\ve[f],
\]
we have already dealt with $\we$. Now we identify $\re$ as the Leray projector of $\we$ on $\Pie$:

\begin{lemma}\label{lem.decomp} 
With the above definition, for any $\alpha>0$, $\mu\in [0,1]$ and $\varepsilon>0$, $\re[f]$ is the Leray projector of $\we[f]$:
$$\re[f] = \mathbb P^\eps(\we[f]).$$
\end{lemma}
\begin{proof}
Any $u$ can be decomposed as $u=v+\nabla p$, where $v=\mathbb P^\eps(u)$ is the Leray projector on $\Pie$, i.e. the unique vector satisfying
$$\left\{\begin{array}{rcll}
\div v&=&0,&\mbox{in }\Pie\\
\curl v&=&\curl u,&\mbox{in }\Pie\\
v\cdot\bn &=&0,&\mbox{on }\partial\Pie\\
\oint_{\partial\Kea{i,j}}v\cdot\tau\d s&=&\oint_{\partial\Kea{i,j}}u\cdot\tau\d s,&\mbox{for any }j\in\{1,\ldots,\Ny\},i\in\{1,\ldots,\Nx\}.
\end{array}\right.$$
In our case, we have according to \eqref{eq.ve} and  \eqref{eq.ue}:
$$\left\{\begin{array}{rcll}
\div\re[f]&=&0&\mbox{in }\Pie\\
\curl\re[f]&=&f-\curl \ve[f] =\curl\we[f]&\mbox{in }\Pie\\
\re[f]\cdot\bn&=&0&\mbox{on }\partial\Pie\\
\oint_{\partial\Kea{i,j}}\re[f]\cdot\tau\d s&=&
\oint_{\partial\Kea{i,j}}\we[f]\cdot\tau\d s&\mbox{for any }j,i.
\end{array}\right.$$
The last equality comes from the equality $\re[f]= \ueps[f] - \Kr{f1_{\Pie}} +\we[f]$ and the Green formula 
$$\oint_{\partial\Kea{i,j}}\Kr{f1_{\Pie}} \cdot\tau\d s =\int_{\Kea{i,j}}f1_{\Pie}=0,$$ 
because $1_{\Pie}$ is the characteristic function on $\Pie$. The uniqueness of 
the decomposition yields the Lemma. 
\end{proof}

The convergence of $\re[f]$ is now obvious. Indeed, we recall that the Leray projector is orthogonal for the $\sL^2$ norm, then for any $\alpha$, $\mu$, $\varepsilon$ and $f$ we have:
\[
\| \re[f] \|_{\sL^2(\Pie)} \leq \| \we [f] \|_{\sL^2(\Pie)} \leq \| \we [f] \|_{\sL^2(\R^2)}.
\]
So, extending $\ueps[f]$ by zero inside the inclusions, we deduce directly from Theorem \ref{theo we}:
\begin{theorem}\label{theo ue}
We recall that  $\alpha_c(\mu)=2-\mu$.
Let $\mu\in [0,1]$ and $\alpha\in (0,\alpha_{c}(\mu))$ be fixed. Then
\begin{equation*}
\|\ueps[f] - \Kr{f1_{\Pie}} \|_{\sL^2(\R^2)}\to 0\quad \mbox{ as }\eps\to 0,
\end{equation*}
uniformly in $f$ verifying
\[ \| f \|_{\sL^1\cap \sL^\infty(\R^2)} \leq M_{0}.\] 
\end{theorem}

\section{Proof of the main Theorem}\label{sect conclusion}

The way to conclude comes from \cite{Lac-curve}  and we write the main steps for a sake of completeness. In general the Sobolev and Lebesgue spaces are considered in the full plane, and $(\ueps,\ome)$  are extended by zero in the obstacles. 
In all this section, we fix $\mu \in [0,1]$  and $\alpha\in (0,\alpha_{c}(\mu))$.

\subsection{Weak convergence of the vorticity}

Thanks to the transport equation \eqref{est transport}, extracting a subsequence, we have that
\[ \ome \rightharpoonup \omega\quad\text{ weak-$*$ in }\sL^\infty(\R^+; \sL^1\cap \sL^\infty(\R^2)),\]
which establishes the point {\it (b)} of Theorem \ref{main theo}, up to a subsequence.

We introduce
\[
M_{0}:= \max \{ \| \omega_{0} \|_{\sL^1(\R^2)},  \| \omega_{0} \|_{\sL^\infty(\R^2)}\},
\]
hence for any $t$ and $\eps$
\begin{equation}\label{ome M0}
\| \ome(t, \cdot) \|_{\sL^1\cap \sL^\infty(\R^2)} \leq M_{0}.
\end{equation}

\subsection{Strong convergence of the velocity}

First we begin by a temporal estimate. 

\begin{lemma}
 There exists a constant $C$ independent of $\varepsilon$ and $t$ such that
\[\|  \partial_t \ome \|_{\sH^{-1}(\R^2)} \leq C.\]
\end{lemma}

\begin{proof} For any $\varepsilon>0$, as $\ueps$ is regular enough and tangent to the boundary,  we can write the equation verified by  $\ome$ for any test function $\varphi \in \sH^1(\R^2)$:
\begin{eqnarray*}
(\partial_t \ome,\varphi)_{\sH^{-1} \times \sH^1} 
=\int_{\Pie}  \ueps \ome \cdot \nabla \varphi = \int_{\R^2}  (\ueps- \Kr{\ome})\ome \cdot \nabla \varphi + \int_{\R^2}  \Kr{\ome}\ome \cdot \nabla \varphi,
\end{eqnarray*}
which is bounded by $C\|\nabla \varphi\|_{\sL^2}$ for the following reason. According to \eqref{ome M0}, Theorem \ref{theo ue} states that $\ueps- \Kr{\ome}$ is uniformly bounded in $\sL^2(\R^2)$ which gives the estimate for the first right hand side term.
For the second term,  we know from \eqref{est biot} and \eqref{ome M0} that  $\Kr{\ome}$ is uniformly bounded whereas $\ome$ is uniformly bounded in $\sL^2$.  
It gives the desired estimates in $\sH^{-1}$.
\end{proof}

\begin{lemma}\label{vorticity limit}
There exists a subsequence of $\ome$ (again denoted by $\ome$) such that $\ome(t,\cdot)\rightharpoonup \omega(t,\cdot)$  in  weak-$\sL^4(\R^2)$ and  in 
 weak-$\sL^{\frac43}(\R^2)$ for all $t$.
\end{lemma}

\begin{trivlist}
                       \item[]\hspace{0cm}{\bf Sketch of proof: }
                       The proof of this lemma is done in \cite[Prop. 5.2]{Lac-curve}. The idea is the following: by Banach-Alaoglu's theorem, we can extract, for each $t$, a subsequence such that  $\ome(t,\cdot)\rightharpoonup \omega(t,\cdot)$  in 
 weak-$\sL^4(\R^2)$  and in  weak-$\sL^{\frac43}(\R^2)$, but the subsequence depends on the time $t$, whereas we want a common sequence for each $t$. For that, we choose by diagonal extraction a common sequence for each $t\in \mathbb{Q}$. Next, for any test function in $\C^\infty_0(\R^2)$ and thanks to the time estimate of the previous lemma, we prove that the sequence works for all $t$. The desired result is obtained by the density of $\C^\infty_0(\R^2)$ in $\sH^1(\R^2)$.  \hspace{0cm} \hfill $\blacksquare$
                     \end{trivlist}

Now, defining $u:= \Kr{\omega}$, we use this subsequence to pass to the limit in the decomposition
\begin{equation}\label{decomp2}
\ueps-u = (\ueps- \Kr{\ome}) + \Kr{\ome-\omega}.
\end{equation}

\begin{theorem}\label{theo : velocity}
We have $\ueps \to u$ strongly in $\sL^2_{\loc}(\R^+ \times \R^2)$,
with $u=\Kr{\omega}$.
\end{theorem}

\begin{proof} The first  term on the right-hand side 
  of  \eqref{decomp2} converges uniformly in time to zero in $\sL^2(\R^2)$ (see Theorem \ref{theo ue} and \eqref{ome M0}). Then the dominated convergence theorem gives the limit in $\sL^2_{\loc}(\R^+ \times \R^2)$.

Concerning the last term: for $x$ fixed, the map $y\mapsto \frac{(x-y)^\perp}{|x-y|^2}$ belongs to  $\sL^{4/3}(B(x,1))\cap \sL^{4}(B(x,1)^c)$, then Lemma \ref{vorticity limit} implies 
that  for all $t,x$, we have 
$$\int_{\R^2}  \frac{(x-y)^\perp}{|x-y|^2} (\ome-\omega)(t,y)\d y \to 0\quad \text{ as }\varepsilon\to 0.$$ So, this integral converges pointwise to zero, and it is uniformly bounded by \eqref{est biot} with respect of $x$ and $t$. Applying   the dominated convergence theorem, we obtain the convergence of $\Kr{\ome-\omega}$ in $\sL^2_{\loc}(\R^+ \times \R^2)$.
This ends the proof.
\end{proof}

This theorem gives the point {\it(a)} of Theorem \ref{main theo}, up to a subsequence.

\subsection{Passing to the limit in the Euler equations}
The purpose of the rest of this section is to prove that $(u,\omega)$ is the unique solution of the Euler equations in $\R^2$.
\begin{theorem}
The pair $(u,\omega)$ obtained is a weak solution of the Euler equations in $\R^2$.
\end{theorem}
\begin{proof}
The divergence and curl conditions  are  verified by the expression: $u=\Kr{\omega}$.

Next, we use that $\ueps$ and $\ome$ verify  $\eqref{eq.vort}$ in the sense of distribution in $\Pie$ and the fact that $\ueps$ is regular and tangent to the boundary, to infer that for any test function $\varphi\in \C^\infty_0([0,\infty)\times \R^2)$, we have
\[
\int_0^\infty\int_{\R^2} \varphi_t \ome \d x\d t+\int_0^\infty\int_{\R^2} \nabla\varphi \cdot  \ueps \ome \d x \d t= -\int_{\R^2}\varphi(0,x)\omega_0(x)1_{\Pie}\d x,
\]
because we have extended $\ome$ by zero and set $\ome(0,\cdot)=\omega_{0}1_{\Pie}$.
By passing to the limit as $\varepsilon\to 0$, thanks to the strong-weak convergence of the pair $(\ueps,\ome)$, we conclude that $(u,\omega)$ verifies the vorticity equation. In the full plane, this is equivalent to state that $u$ verifies the velocity equation.
\end{proof}

All the  results of this section state that for any sequence $\varepsilon_{k}\to 0$, we can extract a subsequence such that $(\ueps,\ome)$ converges to $(u,\omega)$, which is a global weak solution to the 2D Euler equations in the full plane, and where $\omega$ belongs to $\sL^\infty(\R^+;\sL^1\cap\sL^\infty(\R^2))$. Such a solution is unique by the celebrated Yudovich's work \cite{yudo}. Therefore, this solution is the strong solution with initial datum $\omega_{0}$, and we deduce from the uniqueness that the convergences hold without extracting a subsequence. This ends the proof of Theorem \ref{main theo}.

\bigskip

\noindent
{\bf Acknowledgement:} The authors are grateful to Thibaut Deheuvels and  Vincent Munnier for references concerning quasiconvexity (see Lemma \ref{lem.TT-1}).

For this work, the first author is supported by ANR project {\sc Aramis} n$^{\rm o}$ ANR-12-BS01-0021. The second author is partially supported by the Project ``Instabilities in Hydrodynamics'' funded by Paris city hall (program ``Emergences'') and the Fondation Sciences Math\'ematiques de Paris. The third author is partially 
  supported by 
NSF grant DMS-1211806.

\def\cprime{$'$}

\adrese

\end{document}